\documentclass[11pt,a4paper,eqno]{amsart}
\usepackage[left=0.9in, right=0.9in, top=1.5in, bottom=1.5in]{geometry}
\usepackage[utf8]{inputenc}
\usepackage[T1]{fontenc} 
\usepackage[english]{babel}
\usepackage{yhmath}
\usepackage{amsmath,amssymb}
\usepackage{stmaryrd}
\usepackage[
colorlinks=true,
urlcolor=teal,
citecolor=magenta,
linkcolor=teal,
linktocpage,
pdfpagelabels,
bookmarksnumbered,
bookmarksopen,
breaklinks=true]
{hyperref}
\usepackage{amsthm}
\usepackage{tikz}
\usepackage{verbatim}

\usepackage{graphicx}
\usepackage{enumitem}
\usepackage{footnote} %
\usepackage{appendix}

\newtheorem{theorem}{Theorem}[section]

\newtheorem{corollary}[theorem]{Corollary}
\newtheorem{lemma}[theorem]{Lemma}

\newtheorem{proposition}[theorem]{Proposition}
\newtheorem{definition}[theorem]{Definition}
\newtheorem{remark}[theorem]{Remark}
\newtheorem{open}[theorem]{Open Problem}
\newtheorem*{mtheorem1}{Main Theorem 1}
\newtheorem*{mtheorem2}{Main Theorem 2}
\newtheorem*{mtheorem3}{Main Theorem 3}
\newtheorem*{mtheorem4}{Main Theorem 4}
\newtheorem*{theorem*}{Theorem}

\newcommand{\R}{\mathbb{R}}

\newcommand{\N}{\mathbb{N}}
\renewcommand{\H}{\mathcal{H}}

\newcommand{\ds}{\displaystyle}
\newcommand{\rmnote}[1]{}

\usepackage{braket}
\usepackage{esint}
\newcommand{\abs}[1]{{\left|#1\right|}}

\title[Optimality and stability of the radial shapes for the Sobolev trace constant]{Optimality and stability of the radial shapes for the Sobolev trace constant}
\author{Simone Cito$^1$}

\address{$^1$Dipartimento di Matematica e Fisica \lq\lq E. De Giorgi\lq\lq, Università del Salento, Via per Arnesano, 73100 Lecce, Italy.}
\email{simone.cito@unisalento.it}

\begin{document}

\begin{abstract}
In this work we establish the optimality and the stability of the ball for the Sobolev trace operator $W^{1,p}(\Omega)\hookrightarrow L^q(\partial\Omega)$ among convex sets of prescribed perimeter for any $1< p <+\infty$ and $1\le q\le p$. More precisely, we prove that the trace constant $\sigma_{p,q}$ is maximal for the ball and the deficit is estimated from below by the Hausdorff asymmetry. With similar arguments, we prove the optimality and the stability of the spherical shell for the Sobolev exterior trace operator $W^{1,p}(\Omega_0\setminus\overline{\Theta})\hookrightarrow L^q(\partial\Omega_0)$ among open sets obtained removing from a convex set $\Omega_0$ a suitably smooth open hole $\Theta\subset\subset\Omega_0$, with $\Omega_0\setminus\overline{\Theta}$ satisfying a volume and an outer perimeter constraint.

\end{abstract}

\maketitle

\tableofcontents

\section{Introduction}
Let $d\ge 2$, $\Omega\subset\R^d$ be a bounded Lipschitz domain and fix $p>1$. Let us denote by $p^*$ the \emph{critical trace exponent}, that is $p^*=p(d-1)/(d - p)$ if $p<d$ and $p=+\infty$ if $p\ge d$. Let $1\le q<p^*$. The best constant in the Sobolev trace inequality is the largest number $\sigma_{p,q}(\Omega)$ such that
$$
\sigma_{p,q}(\Omega)\|u\|_{L^q(\partial\Omega)}\le\|u\|_{W^{1,p}(\Omega)}
$$
for any $u\in W^{1,p}(\Omega)$; equivalently, $\sigma_{p,q}(\Omega)$ is the reciprocal of the norm of the usual trace operator $T^{p,q}_{\Omega}:W^{1,p}(\Omega)\hookrightarrow L^q(\partial\Omega)$. In view of the compactness of the embedding, the quantity $\sigma_{p,q}(\Omega)$ can be expressed in a variational way as follows 
\begin{equation}
\label{pq_min}		
\sigma_{p,q}(\Omega)=\min_{\substack{\psi\in W^{1,p}(\Omega)\\ \psi\not \equiv0}}\dfrac{\|\psi\|_{W^{1,p}(\Omega)}}{\|\psi\|_{L^q(\partial\Omega)}}=\min_{\substack{\psi\in W^{1,p}(\Omega)\\ \psi\not \equiv0}}\dfrac{\ds\left(\int _{\Omega}|\nabla \psi|^p\;dx+\ds\int_{\Omega}|\psi|^p\;dx\right)^{1/p}}{\left(\ds\int_{\partial \Omega}|\psi|^q \;d\mathcal{H}^{d-1}\right)^{1/q}}.
\end{equation}
If $w\in W^{1,p}(\Omega)$ is a minimizer of \eqref{pq_min}, then it satisfies
\begin{equation}
\label{pq_prob}
\begin{cases}
	-\Delta_p w+|w|^{p-2}w=0 & \mbox{in}\ \Omega\vspace{0.2cm}\\
|\nabla w|^{p-2}\dfrac{\partial w}{\partial \nu}=\lambda|w|^{q-2}w&\mbox{on}\ \partial \Omega\vspace{0.2cm},
\end{cases}
\end{equation}
for some $\lambda>0$ depending on $\|w\|_{L^q(\partial\Omega)}$, where we denote by $ \nu $  the unit  outer normal to the boundary; if $\|w\|_{L^q(\partial\Omega)}=1$, $\lambda=\sigma^p_{p,q}(\Omega)$. 

A minimizer of \eqref{pq_min} or, equivalently, a solution for \eqref{pq_prob}, exists for any $q\in[1,p^*[$ in view of the compactness of the tace operator $T^{p,q}_\Omega$.
When $p=q$, we reduce, in fact, to a Steklov-type nonlinear eigenvalue problem. To be clearer, denoting $\sigma_{p}(\Omega):=\sigma_{p,p}(\Omega)$, we have
\begin{equation}
\label{eig_min}		
\sigma_{p}(\Omega)=\min_{\substack{\psi\in W^{1,p}(\Omega)\\ \psi\not \equiv0}}\dfrac{\|\psi\|_{W^{1,p}(\Omega)}}{\|\psi\|_{L^p(\partial\Omega)}}=\min_{\substack{\psi\in W^{1,p}(\Omega)\\ \psi\not \equiv0}}\left(\dfrac{\ds\int _{\Omega}|\nabla \psi|^p\;dx+\ds\int_{\Omega}|\psi|^p\;dx}{\ds\int_{\partial \Omega}|\psi|^p \;d\mathcal{H}^{d-1}}\right)^{1/p}.
\end{equation}
In this case, any minimizer $w\in W^{1,p}(\Omega)$ of \eqref{eig_min} satisfies
\begin{equation}
\label{eig_prob}
\begin{cases}
	-\Delta_p w+|w|^{p-2}w=0 & \mbox{in}\ \Omega\vspace{0.2cm}\\
|\nabla w|^{p-2}\dfrac{\partial w}{\partial \nu}=\lambda|w|^{p-2}w&\mbox{on}\ \partial \Omega\vspace{0.2cm},
\end{cases}
\end{equation}
for some $\lambda>0$ depending on $\|w\|_{L^p(\partial\Omega)}$; if $\|w\|_{L^p(\partial\Omega)}=1$, $\lambda=\sigma^p_{p}(\Omega)$.


Our aim is to optimize $\sigma_{p,q}(\cdot)$ in a a suitable class of sets, focusing on the trace term in the choice of the constraint. In particular, we prove that the ball is maximal among convex sets with prescribed perimeter. In addition, we provide a quantitative enhancement of the result in terms of an appropriate asymmetry functional, i.e. a non negative functional $\Omega\mapsto\alpha(\Omega)$ vanishing only if $\Omega$ itself is a ball.

The main results of the first part of the paper are the following. The first is the maximality of the ball for $\sigma_{p,q}$.

\begin{mtheorem1}[Isoperimetric inequality for convex sets]\label{teo_1}
Let $p\in]1,+\infty[$ and $q\in[1,p]$, let $\Omega\subset\R^d$ be convex and $\Omega^\star$ be the ball having the same perimeter as $\Omega$.
Then
\begin{equation}\label{ineq_eig}
\sigma_{p,q}(\Omega^\star)\ge\sigma_{p,q}(\Omega),
\end{equation}
with the equality holding if and only if $\Omega$ is a ball.
\end{mtheorem1}

The second result is the stability of inequality \eqref{ineq_eig}. We slightly modify the notation from the previous statement in order to emphasise the dependence of the constants on the size of the involved sets.

\begin{mtheorem2}[Stability of the ball]\label{teo_2}
Let $0<R<+\infty$, $p\in]1,+\infty[$ and $q\in[1,p]$. There exist two positive constants $C(p,q,d,R)$ and $\delta_0(p,q,d,R)$ such that, for every convex set $\Omega$ with $P(\Omega)=P(B_R)$, if $\sigma_{p,q}(B_R)-\sigma_{p,q}(\Omega)\leq \delta_0$, then $\Omega$ is nearly spherical and
\begin{equation}
\label{stab_ineq_intro2}
\sigma_{p,q}(B_R)-\sigma_{p,q}(\Omega)\geq C g(\mathcal{A}^\star_{\mathcal{H}}(\Omega)),
\end{equation}
    where $\mathcal{A}^\star_{\mathcal{H}}$ is the Hausdorff asymmetry index defined in \eqref{eq:ahstar} and $g$ is defined by
    \begin{equation}
    \label{function_g}
        g(s)=
        \begin{cases}
            s^2 & \text{ if } d=2\\
            f^{-1}(s^2) & \text{ if } d=3\\
            s^{\frac{d+1}{2}} & \text{ if } d\geq 4
        \end{cases}
    \end{equation}
    with $f(t)= \sqrt{t \log(\frac{1}{t})}$ for $0 < t < e^{-1}$.
\end{mtheorem2}

The quantitative inequalities for the classical Steklov problem have been studied, e.g., in \cite{brasco2012spectral,gavitone2020quantitative}. The stability of the ball for the least eigenvalue $\sigma_p^p$ of the Steklov-type problem \eqref{eig_prob} has been treated among convex nearly spherical sets sets in the linear case $p=2$ in \cite{ferone2015conjectured,cito2021quantitative}.

\medskip

Actually, the aim of the paper is twofold. Indeed, we ask ourselves whether analogous results can be achieved for a similar functional defined on domains with holes. Let us consider $\Omega_0\subset\R^d$ open, bounded and convex and $\Theta\subset\subset\Omega_0$ a nonempty open set smooth enough to have a normal unit at least in a weak sense. We consider the functional
\begin{equation}
\label{pqbuco_min}		
\tilde{\sigma}_{p,q}(\Omega_0\setminus\overline{\Theta})=\min_{\substack{\psi\in W^{1,p}(\Omega_0\setminus\overline{\Theta})\\ \psi\not \equiv0}}\dfrac{\|\psi\|_{W^{1,p}(\Omega_0\setminus\overline{\Theta})}}{\|\psi\|_{L^q(\partial\Omega_0)}}=\min_{\substack{\psi\in W^{1,p}(\Omega_0\setminus\overline{\Theta})\\ \psi\not \equiv0}}\dfrac{\ds\left(\int _{\Omega_0\setminus\overline{\Theta}}|\nabla \psi|^p\;dx+\ds\int_{\Omega_0\setminus\overline{\Theta}}|\psi|^p\;dx\right)^{1/p}}{\left(\ds\int_{\partial \Omega_0}|\psi|^q \;d\mathcal{H}^{d-1}\right)^{1/q}}.
\end{equation}
The positive quantity $\tilde{\sigma}_{p,q}(\Omega_0\setminus\overline{\Theta})$ is again linked to a Sobolev trace inequality. More precisely, we deal with the \emph{exterior} trace operator $\tilde{T}^{p,q}_{\Omega_0\setminus\overline{\Theta}}:W^{1,p}(\Omega_0\setminus\overline{\Theta})\hookrightarrow L^q(\partial\Omega_0)$, mapping a function in $W^{1,p}(\Omega_0\setminus\overline{\Theta})$ onto its trace \emph{only} on $\partial\Omega_0$. In this framework, $\tilde{\sigma}_{p,q}(\Omega_0\setminus\overline{\Theta})$ is the largest number such that
$$
\tilde{\sigma}_{p,q}(\Omega_0\setminus\overline{\Theta})\|u\|_{L^q(\partial\Omega_0)}\le\|u\|_{W^{1,p}(\Omega_0\setminus\overline{\Theta})}
$$
for any $u\in W^{1,p}(\Omega_0\setminus\overline{\Theta})$, i.e. it is the reciprocal of the operator norm of $\tilde{T}^{p,q}_{\Omega_0\setminus\overline{\Theta}}$. The minimizers of \eqref{pqbuco_min} are weak solutions of the nonlinear problem
\begin{equation}
\label{pqbuco_prob}
\begin{cases}
	-\Delta_p w+|w|^{p-2}w=0 & \mbox{in}\ \Omega_0\setminus\overline{\Theta}\vspace{0.2cm}\\
|\nabla w|^{p-2}\dfrac{\partial w}{\partial \nu}=\lambda|w|^{q-2}w&\mbox{on}\ \partial \Omega_0\vspace{0.2cm}\\
|\nabla w|^{p-2}\dfrac{\partial w}{\partial \nu}=0&\mbox{on}\ \partial \Theta\vspace{0.2cm},
\end{cases}
\end{equation}
with $\lambda>0$ depending again on the trace norm $\|w\|_{L^q(\partial\Omega_0)}$, with $\lambda=\tilde{\sigma}^p_{p,q}(\Omega_0\setminus\overline{\Theta})$ if $\|w\|_{L^q(\partial\Omega_0)}=1$. 

Let us complete this overview with the case $p=q$. Setting $\tilde{\sigma}_{p}(\Omega_0\setminus\overline{\Theta}):=\tilde{\sigma}_{p,p}(\Omega_0\setminus\overline{\Theta})$, we have the following variational characterization
\begin{equation}
\label{eig_minbuco}		
\tilde{\sigma}_{p}(\Omega_0\setminus\overline{\Theta})=\min_{\substack{\psi\in W^{1,p}(\Omega)\\ \psi\not \equiv0}} \left(\dfrac{\ds\int _{\Omega_0\setminus\overline{\Theta}}|\nabla \psi|^p\;dx+\ds\int_{\Omega_0\setminus\overline{\Theta}}|\psi|^p\;dx}{\ds\int_{\partial \Omega_0}|\psi|^p \;d\mathcal{H}^{d-1}}\right)^{1/p} 
\end{equation}
and minimizers solve the nonlinear Steklov-Neumann nonlinear eigenvalue problem
\begin{equation}
\label{eig_probbuco}
\begin{cases}
	-\Delta_p w+|w|^{p-2}w=0 & \mbox{in}\ \Omega_0\setminus\overline{\Theta}\vspace{0.2cm}\\
|\nabla w|^{p-2}\dfrac{\partial w}{\partial \nu}=\lambda|w|^{p-2}w&\mbox{on}\ \partial \Omega_0\vspace{0.2cm},\\
|\nabla w|^{p-2}\dfrac{\partial w}{\partial \nu}=0&\mbox{on}\ \partial \Theta\vspace{0.2cm}
\end{cases}
\end{equation}
where $\lambda=\tilde{\sigma}^p_{p}(\Omega_0\setminus\overline{\Theta})$ in case of normalization $\|w\|_{L^p(\partial\Omega_0)}=1$.

The maximality of the spherical shell is proved in the following

\begin{mtheorem3}[Isoperimetric inequality for holed sets]\label{teo_3}
Let $p\in]1,+\infty[$, $q\in[1,p]$ and let $\Omega_0\subset\R^d$ be convex and  $\Theta\subset\subset\Omega_0$ an open set sufficiently small for the weak well posedness of Problem \eqref{pqbuco_prob}. Let $A_{R_1,R_2}$ be the spherical shell such that $P(\Omega_0)=P(B_{R_2})$ and $|\Omega_0\setminus\overline{\Theta}|=|A_{R_1,R_2}|$.
Then
\begin{equation}
\label{bucostab_ineq_intro1}
\tilde{\sigma}_{p,q}(A_{R_1,R_2})\ge\tilde{\sigma}_{p,q}(\Omega_0\setminus\overline{\Theta}),
\end{equation}
with the equality holding if and only if $\Omega_0\setminus\overline{\Theta}=A_{R_1,R_2}$.
\end{mtheorem3}

Once approached the study of the isoperimetric result above, we asked ourselves about a possible quantitative enhancement of the previous theorem. We started our analysis referring to \cite{cpp}. That paper, as far as we are aware, provides the first quantitative result within the setting of domains featuring holes where \emph{both} the outer and inner boundaries are perturbed independently. In that research an asymmetry functional for sets of the form $\Omega_0\setminus\overline{\Theta}$ has been introduced. This functional is designed to handle two distinct boundary perturbations, with $\partial\Omega_0$ and $\partial\Theta$ governed by different boundary conditions and with $\Omega_0\setminus\overline{\Theta}$ required to meet specific constraints (namely, the outer perimeter and the volume). We will introduce the so-called \emph{hybrid} asymmetry $\alpha_{hyb}(\cdot)$ in Definition \ref{inner_asymmetry_def} and delve into the motivation of that choice in Section 2.

Let us introduce the framework for the stability result. Let $0<R_1<R_2<+\infty$ and denote by $\vartheta_{R_1,R_2}$ a positive constant such that $\vartheta_{R_1,R_2}\le\min\{R_1,R_2-R_1\}/2$.  We set
\begin{equation}
\label{admissiblesets}
\begin{split}\mathcal T_{R_1,R_2}= \{\Omega\subseteq\R^d :&\ \Omega=\Omega_0\setminus\overline{\Theta},\  \Omega_0,\Theta\ \text{convex},\ |\Omega|=|A_{R_1,R_2}|,\ P(\Omega_0)=P(B_{R_2}),\\ & d_\mathcal{H}(\Theta,\Omega_0)\ge\vartheta_{R_1,R_2}\ \text{and}\ \rho(\Theta)\ge\vartheta_{R_1,R_2}\},
\end{split}
\end{equation}
 where  $\rho(\cdot)$ is the inradius and $d_{\mathcal H}(\cdot, \cdot )$ is the Hausdorff distance. We prove the following stability result in $\mathcal{T}_{R_1,R_2}$.

\begin{mtheorem4}[Stability of the spherical shell]\label{teo_4}
Let $0<R_1<R_2<+\infty$, $p\in]1,+\infty[$ and $q\in[1,p]$. There exist two positive constants $C(p,q,d,R_1,R_2)$ and 
$\delta_0(p,q,d,R_1,R_2)$ such that, for every open set $\Omega_0\setminus\overline{\Theta}\in\mathcal{T}_{R_1,R_2}$, if $\tilde{\sigma}_{p,q}(A_{R_1,R_2})-\tilde{\sigma}_{p,q}(\Omega_0\setminus\overline{\Theta})\leq \delta_0$, then $\Omega_0\setminus\overline{\Theta}$ is $(R_1,R_2)$-nearly annular and
\begin{equation}
\label{bucostab_ineq_intro2}
\tilde{\sigma}_{p,q}(A_{R_1,R_2})-\tilde{\sigma}_{p,q}(\Omega_0\setminus\overline{\Theta})\geq C \alpha_{hyb}(\Omega_0\setminus\overline{\Theta}),
\end{equation}  
where $\alpha_{hyb}(\cdot)$ is the hybrid asymmetry functional in Definition \ref{inner_asymmetry_def}.
\end{mtheorem4}

The study of $\tilde{\sigma}_{p,q}$ is mainly motivated by the following facts. On one hand, the problem of the optimal constant for the exterior trace has been widely studied imposing a \emph{Dirichlet} condition on $\partial\Theta$ (see, e.g., \cite{BGR}), so we found reasonable to analyse what happens replacing it with a \emph{Neumann} condition on the hole. On the other hand, $\tilde{\sigma}_{p,q}$ is linked to the nonlinear Robin-Neumann eigenvalue with negative boundary parameter $\tilde{\lambda}_p(\cdot,\beta)$.  This latter admits $A_{R_1,R_2}$ as unique maximizer (see \cite{paoli2020sharp} for the proof), but its stability (conjectured in \cite[Open Problem 5.2]{cpp}) remained an open problem. We obtain this quantitative result as an adaptation of the proof of Main Theorem 4. Regarding the case $p=q$ (and in particular $p=q=2$), a number of authors have recently been exploring the Laplace and $p$-Laplace eigenvalue problem in the scenario where the outer and inner boundaries are subject to different conditions. We cite, first of all, the pioneeristic paper \cite{payne1961some}, which paved the way to study this kind of mixed problems. Among the others, considering first the boundary condition on $\partial\Omega_0$ and then on $\partial\Theta$, we cite the Robin-Neumann case in \cite{paoli2020sharp}, the Neumann-Robin case (with positive boundary parameter) \cite{dellapietra2020optimal,hersch1962contribution}, the Dirichlet-Neumann case \cite{anoop2023domain,anoop2021shape},the Dirichlet-Dirichlet case\cite{biswas2021optimization,chorwadwala2022optimal},
the Steklov-Dirichlet case \cite{ftouhi2022place,
gavitone2021isoperimetric,gavitone2024monotonicity, hong2022first, paoli2021stability,
verma2020eigenvalue} and the Steklov-Robin case \cite{gavitone2023steklov}. Recently, a comparison Talenti type result has been studied in the same class of sets, see \cite{barbato2025talenticomparisonresultssolutions}.

On the other hand, the quantitative versions of the spectral inequalities on holed domain are still in the early stages of development. We mention \cite{paoli2021stability}, where a result in this direction is proved for the first Steklov-Dirichlet Laplacian eigenvalue among the class of holed domains where the inner hole $\Theta$ is a given ball well cointained in $\Omega_0$. The inner ball can be only translated and so the only perturbation of the boundary acts on $\partial\Omega_0$. Up to our knowledge, the first result in which the perturbation acts independently on both boundaries is presented in \cite{cpp}. Recently, in \cite{amabar}, the stability of the annulus for the torsion of multiply connected domains has been given in terms of the Fraenkel Asymmetry.  

We point out that for both the convex and the holed case, the results depend on the size of the admissible sets; this is evident in particular in the stability results, since the constants in Main Theorem 2 and in Main Theorem 4 are given in terms of the radius of the ball and in terms of the radii of the spherical shell, respectively. Even the fact that we consider $q$ only in $[1,p]$ and not on the whole compactness interval $[1,p^*[$ depends on a technical issue deeply linked to the size of the admissible sets. Indeed, we will see in Proposition \ref{pro:radball} and Proposition \ref{pro:radbuco} that problems \eqref{pq_prob} and \eqref{pqbuco_prob} have radial solution respectively on $B_R$ and $A_{R_1,R_2}$, regardless of the radii, only if $q\in[1,p]$. Now, the proofs of the main results rely on the well established web function method, based on the construction of a suitable test function by dearranging a positive \emph{radial} minimizer for the ball or the spherical shell. We will go into detail about the properties of the minimizers for radial shapes and about the "radiality thresholds" in Section \ref{subsec_radial}. In particular, in Remark \ref{rem:rad} we point out that our choice to present the results for $p\in]1,+\infty[$ and $q\in[1,p]$ regardless of the radii is motivated by our shape-optimization focused perspective, even if the main theorems hold in a more general range of parameters.

For what concerns the stability of the exterior trace constant for the domains with holes, the main difference from \cite{cpp} is that here we do not have explicit minimizers as  in the linear case. More specifically, the eigenfunction for $\tilde{\sigma}^2_2$ is a linear combination of two Bessel functions, whereas in the general nonlinear case we do not have an explicit representation of the minimizing function. Then, in order to apply our dearrangement technique, it is needed to check carefully all the functional properties (radial simmetry and radial monotonicity) required. 
Moreover, as already noticed in \cite{cpp}, the inner hole with Neumann boundary condition cannot be treated as in \cite{fuglede1989stability}, since a cancellation of the inner boundary integrals occurs due to the boundary condition. So, it seems that we cannot use the standard Fuglede approach to treat either the asymmetry of the outer boundary or the deviation of the hole. We will go in detail to present the hybrid asymmetry functional: even if it seems far from being optimal and, at a first sight, a bit artificial, it matches with our purposes. Indeed, we were mainly focused on finding the optimal shapes for $\tilde{\sigma}_{p,q}$ and, for a possible quantitative enhancement, we were only looking for an actual asymmetry functional (i.e. a positive quantity equal to zero if and only if $\Omega_0\setminus\overline{\Theta}=A_{R_1,R_2}$)  to quantify the distance from the spherical shell. Even if we spend several pages for its definition, motivation and meaning, $\alpha_{hyb}$ is not the core of our discussion.

We point out that also in this case, since we have to handle two different constrains, the hybrid asymmetry functional is the result of two different terms. The appearence of a mixed asymmertry as a combination of one term per each object seems to be natural in problems in which you can perturb more features independently. In addition to the already cited \cite{cpp}, see also \cite{amato2023talenti}, where the authors combine three asymmetry terms to control, respectively, the symmetry of the domain, the symmetry of the datum and the symmetry of the solution for a Poisson equation. 

\bigskip

The structure of this paper is the following. In Section \ref{notations_sec} we recall some of the basic notation, we state (prove, if new) some useful properties of the solution of the nonlinear problem and recall some geometric tools. In Section \ref{sec:hybrid} we introduce the hybrid asymmetry functional $\alpha_{hyb}(\cdot)$. In Section \ref{main_sec_conv} we provide the proofs of Main Theorems 1 and 2. In Section \ref{main_sec_buco} we prove Main Theorems 3 and 4, going in detail on the arguments to get the quantitative estimate. In Section \ref{rem_sec}, some consequences of Main Theorems 2 and 4 are given in terms of stability of Robin spectral problems with negative boundary parameter; finally, some open problems are discussed.

\section{Notation and Preliminary results}
\label{notations_sec}

\subsection{Notations and basic tools of measure theory}
Throughout this paper, we denote by $|\cdot|$ and $\mathcal{H}^k(\cdot)$, the $d-$dimensional Lebesgue measure and the $k-$dimensional Hausdorff measure in $\mathbb{R}^d$, respectively. The unit open  ball in $\mathbb{R}^d$ will be denoted by $B_1$ and the unit sphere in $\mathbb{R}^d$ by $\mathbb{S}^{d-1}$; we also define $\omega_d:=|B_1|$.  
More generally, we denote with $B_r(x_0)$, that is the ball centered at $x_0$ with radius $r$, and, for any $0<R_1<R_2<+\infty$, we denote by  $A_{R_1,R_2}$ the spherical shell $B_{R_2}\setminus \overline{B_{R_1}}$. Furthermore, we denote by $\chi_{E}$ the characteristic function of a measurable set $E$ and, for any measurable funciton $f\ge 0$, by $Vol_f(E)$ its weighted volume $\int_Ef\:dx$.

When not specified, $\Omega_0\setminus \overline{\Theta}$ will denote a subset of $\mathbb{R}^d$, $d\geq 2$, such that $\Omega_0$ is an open, bounded and convex set and $\Theta$ is an open set such that $\overline{\Theta}\subset \Omega_0$ and sufficiently smooth to define (at least in a weak sense) a nonlinear Neumann boundary condition.

The distance of a point  from the boundary $\partial\Omega_0$ is the function 
\begin{equation*}
d(x)= \inf_{y\in \partial \Omega_0} |x-y|, \quad x\in \overline{\Omega_0},
\end{equation*}
moreover, the inradius of $\Omega_0$ is the radius of the largest ball contained in $\Omega_0$, i.e.
\begin{equation*}
\rho({\Omega_0})=\max \{d(x),\; x\in\overline{\Omega_0}\}.
\end{equation*}
We recall now some basic notions about the Hausdorff distance. If $E,F$ are two convex sets in  $\mathbb{R}^d$, the Hausdorff distance between $E, F$ is defined as
\[
d_{\mathcal{H}}(E,F) := \inf\left\{ \varepsilon>0 \, : \,  E \subset F + \varepsilon B_1 , \, F \subset E + \varepsilon B_1  \right\},
\]
where $F + \varepsilon B_1$ is intended in the sense of the Minkowski sum.

We will denote by $E^\sharp$ and $E^\star$ the balls having respectively the same volume and the same perimeter as $E$.

Since we deal with a perimeter constraint, we recall its rigorous definition.
\begin{definition}
Let $E\subseteq\mathbb{R}^d$ be measurable and let $\Omega\subseteq\mathbb{R}^d$ be open. We define the perimeter of $E$ in $\Omega$ as
$$P(E,\Omega):=\sup\left\{\int_E\text{div}(\varphi)\:dx : \varphi\in C^1_c(\Omega;\mathbb{R}^d),\|\varphi\|_\infty\le 1\right\}$$
and we say that $E$ is of finite perimeter in $\Omega$ if $P(E,\Omega)<+\infty$. If $\Omega=\mathbb{R}^d$ we simply say that $E$ is of finite perimeter and denote its perimeter by $P(E)$.
\end{definition}

Let us recall that, if $E$ is sufficiently regular (e.g. if $E$ is a bounded Lipschitz domain, in particular a convex set), it holds $P(E,\Omega)=\mathcal{H}^{d-1}(\partial E\cap\Omega)$.

An important tool for our proofs is the following.

\begin{theorem}[Coarea formula]
    Let $\Omega\subset\R^d$ be a measurable set, $f:\Omega\to\R$ be a Lipschitz function and let $u:\Omega\to\R$ be a measurable function. Then,
    \begin{equation}
        \label{coarea}
        \int_{\Omega} u(x) \lvert \nabla f(x) \rvert \, dx=\int _{\mathbb {R} }dt\int_{(\Omega\cap f^{-1}(t))}u(y)\, d\mathcal{H}^{d-1}(y).
    \end{equation}
\end{theorem}
To get acquainted about the sets of finite perimeter, the coarea formula and related topics see, for instance, \cite{AFP}.

\subsection{The nonlinear problem: properties of the solutions}
\label{subsec_radial}

In this section we collect some basic properties of the minimizers of \eqref{pq_min} and \eqref{pqbuco_min}. In particular, we will focus on the case when the convex domain for $\sigma_{p,q}$ is a ball and the domain with a hole for $\tilde{\sigma}_{p,q}$ is a spherical shell. We omit the proofs that appear straightforward and present those we deem more significant.

We start with an overview about the case $p=q$, analyzing $\sigma_{p}(\Omega)$ and the minimizers of \eqref{eig_min}; the proofs are not given here as they can be found in \cite{MarRossi}.

\begin{proposition}\label{pro:eigball}
Let $1<p<+\infty$ and let $\Omega\subset\R^d$ be a bounded Lipschitz domain. The following items hold.
\begin{itemize}
\item[(i)] There exists a minimizer $w\in W^{1,p}(\Omega)$ of \eqref{eig_min},  which is a weak solution to \eqref{eig_prob}. 
\item[(ii)]  $\sigma^p_{p}(\Omega)$ is simple, i.e. all the associated eigenfunctions are scalar multiple of each other and do not change sign. Moreover, $\sigma^p_{p}(\Omega)$ is isolated.
\item[(iii)] Let $R>0$ and let $z_p$ be the positive minimizer of problem \eqref{eig_min} on the ball $B_R$. Then, $z_p$ is radially symmetric and increasing, in the sense that $z_p(x)=\Psi_p(|x|)$, with $\Psi_p'(r)>0$ in $]0,R[$.
\end{itemize}
\end{proposition}

In the general case $p\neq q$ the situation is more involved. One would expect the ball to continue to have a radial solution. Instead, a very surprising situation occurs: the radiality of the solution is guaranteed only when $1\le q<p$ regardless of the radius; otherwise, it is not. Those results are proved in \cite{DoTo} for $1<q<p$; for the case $q=1<p$, they can be achieved adapting the proofs of \cite{bragonisp}, where the authors treated the case $p=2$. We collect these results in the following proposition.

\begin{proposition}\label{pro:radball}
Let $1<p<+\infty$ and $R>0$ be fixed. Let $z_p$ be the radial positive normalized minimizer for $\sigma_p(B_R)$ and $1\le q<p^*$.
\begin{itemize}
\item[(i)] If there exists a radial minimizer for $\sigma_{p,q}(B_R)$ then it is a multiple of $z_p$.
\item[(ii)] Assume there exists a radial minimizer for $\sigma_{p,q_0}(B_R)$. If $1\le q<q_0$ then any minimizer for $\sigma_{p,q}(B_R)$ is a multiple of $z_p$.
\item[(iii)] Let $1<q<p$. Then the solution of the boundary value problem \eqref{pq_prob} on $B_R$ is
unique and it is a multiple of $z_p$. In particular any minimizer for \eqref{pq_min} on $B_R$ is
a multiple of $z_p$.
\item[(iv)] If $q=1$, $\sigma_{p,1}(\Omega)$ admits unique minimizer for any Lipschitz set $\Omega$; in particular, the minimizer for $\sigma_{p,1}(B_R)$ is radial.
\end{itemize}
\end{proposition}

If $p<q<p^*$, the radiality of the solution of \eqref{pq_min} on $B_R$ holds only in a certain range of radii. More precisely, the following proposition holds 
\begin{proposition}[\cite{DoTo}, Theorem 2]\label{pro:radballlarge}
Let $1<p<+\infty$ be fixed and let us denote by $z_p$ the solution of \eqref{eig_prob}.
\begin{itemize}
\item[(i)] Let $r> 0$. There exists $Q(r)\ge p$ such that, if $q > Q(r)$, then there is no radial minimizer for $\sigma_{p,q}(B_r)$.
\item[(ii)] Let $p < q < p^*$. There exists $R(q)$ such that, for any $r> R(q)$ there is no
radial minimizer for $\sigma_{p,q}(B_r)$.
\item[(iii)] There exists a function of the radius $\bar{Q}:]0,+\infty[\to[1,p^*[$ 
such that
\begin{align*}
  q\le \bar{Q}(r)\quad&\Rightarrow\quad \text{any minimizer for $\sigma_{p,q}(B_r)$ is a multiple of $z_p$, then it is radial} \\
  q> \bar{Q}(r)\quad&\Rightarrow\quad \text{there is no radial minimizer for $\sigma_{p,q}(B_r)$} 
\end{align*}
Finally, it holds $p\le \bar{Q}(r)\le {Q}(r)$
\end{itemize}
\end{proposition}

\begin{remark}[further radial cases]\label{rem:rad}
In view of Proposition \ref{pro:radballlarge}, the radiality of the solution on the ball still holds when either the radius of the ball or the trace exponent $q$ are sufficiently small. In these further radial cases the proof remains the same as the case $1\le q\le p$, having unfortunately a size restriction for the admissible sets. However, as our primary goal is to present shape optimization results whose validity is regardless on the size of the admissible sets, we will concentrate on the range $1\le q\le p$ and present the results for any positive radius. Nevertheless, we will always keep in mind that the main results of the paper still hold even in the further radial cases where size limitations on the admissible sets exist.
\end{remark}

\medskip

In view of the previous remark, even for the functional $\tilde{\sigma}_{p,q}$, we focus on the cases in which the radiality is guaranteed regardless of the size of $A_{R_1,R_2}$. We are interested in getting the analogous results of Proposition \ref{pro:eigball} and Proposition \ref{pro:radball} for domains with holes. We start with the case $p=q$. We omit the proof since it is a straightforward adaptation of standard arguments; for some refercences, follow, for instance, Section 2.1 in \cite{paoli2020sharp}, in particular Propositions 2.2, 2.3 and 2.5.

\begin{proposition}\label{pro:prop_eigenf_buco}
Let $1<p<+\infty$ let $\Omega_0,\Theta\subset\R^d$ be two Lipschitz domains such that $\Theta\subset\subset\Omega_0$. If $p=q$, the following items hold.
\begin{itemize}
\item[(i)] There exists a minimizer $w\in W^{1,p}(\Omega_0\setminus\overline{\Theta})$ of \eqref{eig_minbuco}, which is a weak solution to \eqref{eig_probbuco}. 
\item[(ii)]  $\tilde{\sigma}^p_{p}(\Omega_0\setminus\overline{\Theta})$ is simple, i.e. all the associated eigenfunctions are scalar multiple of each other and do not change sign. Moreover, $\tilde{\sigma}^p_{p}(\Omega_0\setminus\overline{\Theta})$ is isolated.
\item[(iii)] Let $R>0$ and let $z_p$ be the positive minimizer of problem \eqref{eig_minbuco} on the spherical shell $A_{R_1,R_2}$. Then, $z_p$ is radially symmetric and increasing, in the sense that $z_p(x)=\Psi_p(|x|)$, with $\Psi_p'(r)>0$ in $]R_1,R_2[$.
\end{itemize}
\end{proposition}

If $p\neq q$, the radiality of the solution on $A_{R_1,R_2}$ is guaranteed whenever $1\le q< p$, without any restriction on $R_1$ and $R_2$.

\begin{proposition}\label{pro:radbuco}
Let $1<p<+\infty$ and $R_2>R_1>0$ be fixed. Let $z_p$ be the radial positive normalized minimizer for $\tilde{\sigma}_p(A_{R_1,R_2})$ and $1\le q<p^*$.
\begin{itemize}
\item[(i)] If there exists a radial minimizer for $\tilde{\sigma}_{p,q}(A_{R_1,R_2})$ then it is a multiple of $z_p$.
\item[(ii)] Assume there exists a radial minimizer for $\sigma_{p,q_0}(A_{R_1,R_2})$. If $1\le q<q_0$ then any minimizer for $\sigma_{p,q}(A_{R_1,R_2})$ is a multiple of $z_p$.
\item[(iii)] Let $1<q<p$. Then the solution of the boundary value problem \eqref{pqbuco_prob} is
unique and it is a multiple of $z_p$. In particular any minimizer for \eqref{pqbuco_min} is
a multiple of $z_p$.
\item[(iv)] If $q=1$, $\tilde{\sigma}_{p,1}(\Omega_0\setminus\overline{\Theta})$ admits unique minimizer for any $\Omega_0,\Theta\subset\R^d$ Lipschitz with $\Theta\subset\subset\Omega_0$; in particular, the minimizer for $\tilde{\sigma}_{p,1}(A_{R_1,R_2})$ is radial.
\end{itemize}
\end{proposition}
\begin{proof}
First of all, if the minimizer $v$ for $\tilde{\sigma}_{p,q}(A_{R_1,R_2})$ is radial, then it is constant on $\partial B_{R_2}$. Arguing as in the proof of \cite[Proposition 2]{DoTo}, $v$ is a multiple of $z_p$, proving (i). 

Let us prove (ii) now. Assume that $\tilde{\sigma}_{p,q_0}(A_{R_1,R_2})$ has radial minimizer; in view of (i) it is (a multiple of) $z_p$. Let now $v$ be a minimizer for $\tilde{\sigma}_{p,q}(A_{R_1,R_2})$ and let us prove that $v$ is radial. If $v$ is constant on $\partial B_{R_2}$, then $v$ is a
multiple of $z_p$ in view of the same argument as in the proof of \cite[Proposition 2]{DoTo} and thus is radial. Assume now, by contradiction, that $v$ is not constant on the boundary. The H\"older inequality gives the estimate
$$
\left(\int_{\partial B_{R_2}}v^q\:d\mathcal{H}^{d-1}\right)^{\frac{1}{q}}<P(B_{R_2})^{\frac{q_0-q}{q_0q}} \left(\int_{\partial B_{R_2}}v^{q_0}\:d\mathcal{H}^{d-1}\right)^{\frac{1}{q_0}}
$$
where the inequality is strict since $v$ is assumed not to be constant on $\partial B_{R_2}$. On the other hand, since $z_p$ is constant on $\partial B_{R_2}$
$$
\left(\int_{\partial B_{R_2}}z_p^q\:d\mathcal{H}^{d-1}\right)^{\frac{1}{q}}=P(B_{R_2})^{\frac{q_0-q}{q_0q}} \left(\int_{\partial B_{R_2}}z_p^{q_0}\:d\mathcal{H}^{d-1}\right)^{\frac{1}{q_0}}
$$
Now, $z_p$ is a minimizer for $\tilde{\sigma}_{p,q_0}(A_{R_1,R_2})$ and just a competitor for $\tilde{\sigma}_{p,q}(A_{R_1,R_2})$; conversely, $v$ is a minimizer for $\tilde{\sigma}_{p,q}(A_{R_1,R_2})$ and just a competitor for $\tilde{\sigma}_{p,q_0}(A_{R_1,R_2})$. 
This entails the strict estimate
\allowdisplaybreaks
\begin{align*}
\tilde{\sigma}_{p,q_0}(A_{R_1,R_2})&\le\dfrac{\ds\left(\int _{A_{R_1,R_2}}|\nabla v|^p\;dx+\ds\int_{A_{R_1,R_2}}v^p\;dx\right)^{1/p}}{\left(\ds\int_{\partial B_{R_2}}v^{q_0} \;d\mathcal{H}^{d-1}\right)^{1/q_0}}< P(B_{R_2})^{\frac{q_0-q}{q_0q}} \dfrac{\ds\left(\int _{\Omega}|\nabla v|^p\;dx+\ds\int_{\Omega}v^p\;dx\right)^{1/p}}{\left(\ds\int_{\partial \Omega}v^{q} \;d\mathcal{H}^{d-1}\right)^{1/q}}\\
&=P(B_{R_2})^{\frac{q_0-q}{q_0q}}\tilde{\sigma}_{p,q}(A_{R_1,R_2})=\dfrac{\ds\left(\int_{\partial B_{R_2}}z_p^q\:d\mathcal{H}^{d-1}\right)^{\frac{1}{q}}}{\ds\left(\int_{\partial B_{R_2}}z_p^{q_0}\:d\mathcal{H}^{d-1}\right)^{\frac{1}{q_0}}}\tilde{\sigma}_{p,q}(A_{R_1,R_2})\\
&\le\dfrac{\ds\left(\int_{\partial B_{R_2}}z_p^q\:d\mathcal{H}^{d-1}\right)^{\frac{1}{q}}}{\ds\left(\int_{\partial B_{R_2}}z_p^{q_0}\:d\mathcal{H}^{d-1}\right)^{\frac{1}{q_0}}} \dfrac{\ds\left(\int _{A_{R_1,R_2}}|\nabla z_p|^p\;dx+\ds\int_{A_{R_1,R_2}}z_p^p\;dx\right)^{1/p}}{\left(\ds\int_{\partial B_{R_2}}z_p^{q} \;d\mathcal{H}^{d-1}\right)^{1/q}}\\
&\le\dfrac{\ds\left(\int _{A_{R_1,R_2}}|\nabla z_p|^p\;dx+\ds\int_{A_{R_1,R_2}}z_p^p\;dx\right)^{1/p}}{\left(\ds\int_{\partial B_{R_2}}z_p^{q_0} \;d\mathcal{H}^{d-1}\right)^{1/q_0}}=\tilde{\sigma}_{p,q_0}(A_{R_1,R_2}),
\end{align*}
giving a contradiction. Thus, if $1 \le q < q_0$, any minimizer $v$
for $\tilde{\sigma}_{p,q}(A_{R_1,R_2})$ must be constant on $\partial B_{R_2}$, and so radial and multiple of $z_p$, proving (ii).

Item (iii) is a direct consequence of (i) and (ii). 

Finally, the first part of item (iv) can be proved via standard tools, while the radiality issue is a direct consequence of the uniqueness of the solution.
\end{proof}

\begin{remark}[Immediate geometric upper bounds]\label{trivial} Testing \eqref{pq_min} and \eqref{pqbuco_min} with the characteristic functions of $\Omega$ and $\Omega_0\setminus\overline{\Theta}$ respectively we get
    $$\sigma_{p,q}(\Omega)\le\frac{|\Omega|^{1/p}}{P(\Omega)^{1/q}},\quad \tilde{\sigma}_{p,q}(\Omega_0\setminus\overline{\Theta})\le\frac{|\Omega_0\setminus\overline{\Theta}|^{1/p}}{P(\Omega_0)^{1/q}}.$$
\end{remark}

\begin{remark}[relation between $\sigma_{p,q}$ and $\tilde{\sigma}_{p,q}$]
Let us consider $\Omega_0$ a convex domain and $\Theta\subset\subset\Omega_0$ an admissible hole. Let $u_{\Omega_0}\in W^{1,p}(\Omega_0)$ be a minimizer for \eqref{pq_min} on $\Omega_0$. Using $u_{\Omega_0}|_{\Omega_0\setminus\overline{\Theta}}$ as a test function for $\tilde{\sigma}_{p,q}(\Omega_0\setminus\overline{\Theta})$ one has
\begin{align*}\tilde{\sigma}_{p,q}(\Omega_0\setminus\overline{\Theta})&\le\dfrac{\ds\left(\int _{\Omega_0\setminus\overline{\Theta}}|\nabla u_{\Omega_0}|^p\;dx+\ds\int_{\Omega_0\setminus\overline{\Theta}}|u_{\Omega_0}|^p\;dx\right)^{1/p}}{\left(\ds\int_{\partial\Omega_0\cup\partial\Theta}|u_{\Omega_0}|^q \;d\mathcal{H}^{d-1}\right)^{1/q}}\\
&\le\dfrac{\ds\left(\int _{\Omega_0}|\nabla u_{\Omega_0}|^p\;dx+\ds\int_{\Omega_0}|u_{\Omega_0}|^p\;dx\right)^{1/p}}{\left(\ds\int_{\partial\Omega_0}|u_{\Omega_0}|^q \;d\mathcal{H}^{d-1}\right)^{1/q}}=\sigma_{p,q}(\Omega_0),
\end{align*}
i.e. the trace constant of the outer set $\Omega_0$ is always larger than the exterior trace constant of any of the holed sets of the form $\Omega_0\setminus\overline{\Theta}$. In particular, if $P(\Omega_0)=P(B_R)$, $p>1$ and $1\le q\le p$, in view of the maximality of $B_R$ for $\sigma_{p,q}$ (Main Theorem 1), it holds 
$$\tilde{\sigma}_{p,q}(\Omega_0\setminus\overline{\Theta})\le\sigma_{p,q}(B_R).$$ 
This fact heuristically suggests another reason why, for a nontrivial maximization of $\tilde{\sigma}_{p,q}$, we also need the volume constraint (which even prevents that $\Theta=\emptyset$). Indeed, the problem without the measure constraint
$$\max\{\tilde{\sigma}_{p,q}(\Omega_0\setminus\overline{\Theta}):\Omega_0\ \text{convex},\ P(\Omega_0)=P(B_R),\ \Theta\subset\subset\Omega_0\}$$
attains its maximum on ${\sigma}_{p,q}(B_R)$ in view of the previous observations.
\end{remark}

\begin{remark}[the case $p=q=1$]
The results presented in the paper hold for $p\in]1,+\infty]$ and $q\in[1,p]$. The case $p=q=1$ is different and immediate. The constant $\sigma_1(\Omega)$ has been studied in \cite{AMaRo}. There the authors prove that
$$\sigma_1(\Omega)\le\min\left\{\frac{|\Omega|}{P(\Omega)},1\right\},$$
with the equality holding for some particular sets. Among them, the most relevant example for our purposes is the ball, where equality holds and one has
$$\sigma_1(B_R)=\min\left\{\frac{|B_R|}{P(B_R)},1\right\}=\min\left\{\frac{\omega_dR^d}{d\omega_dR^{d-1}},1\right\}=\min\left\{\frac{R}{d},1\right\}$$
Thus, problem
$$\max\left\{\sigma_1(\Omega):\Omega\subset\R^d\ \text{convex},\ P(\Omega)=d\omega_dR^{d-1}\right\}$$
admits $B_R$ as a trivial solution. Indeed, for any $\Omega\subset\R^d$ convex with $P(\Omega)=d\omega_dR^{d-1}$, it holds $\Omega^\star=B_R$; so $P(\Omega^\star)=P(\Omega)$, $|\Omega|\le|\Omega^\star|$ and one has
$$\sigma_1(\Omega)\le\min\left\{\frac{|\Omega|}{P(\Omega)},1\right\}\le\min\left\{\frac{|\Omega^\star|}{P(\Omega^\star)},1\right\}=\sigma_1(\Omega^\star).$$
If $\sigma_1(\Omega^\star)=\frac{|\Omega^\star|}{P(\Omega^\star)}$, i.e. if the radius $R(\Omega^\star)$ of $\Omega^\star$ is less than the dimension $d$, then the ball $\Omega^\star$ is the only maximizer, since, if $\Omega\neq\Omega^\star$ holds, the second inequality above is strict. Otherwise, if $\sigma_1(\Omega^\star)=1$, the unicity of the maximizer is not assured, since the existence of another possible convex set $\Omega\neq\Omega^\star$ with $\sigma_1(\Omega)=1$ does not seem excluded.

By adapting the same argument as in \cite{AMaRo}, we get the same conclusion for the maximality of $A_{R_1,R_2}$ for $\tilde{\sigma}_1$.
\end{remark}

\subsection{Nearly sperical and nearly annular sets; Hausdorff asymmetry indices}
In this section we introduce some notions related with the closeness to the radial shapes. We start recalling the nearly spherical sets (see \cite{fuglede1989stability,fusco2015quantitative}).
\begin{definition}[nearly spherical sets]
\label{NSset}
Let $\Omega_0\subset \mathbb{R}^d$ be an open bounded set with $0\in \Omega_0$. The set  $\Omega_0$ is said a \emph{$R-$nearly spherical set parametrized by $u$} if there exist a constant  $R>0$ and $u\in W^{1,\infty}(\mathbb{S}^{d-1})$, with $||u||_{W^{1,\infty}}\le R/2$, such that
\begin{equation*}
 \partial    \Omega_0= \left\{y \in \R^d \colon y=\xi  (R+u(\xi)), \ \xi \in \mathbb{S}^{d-1}\right\}.
\end{equation*}
\end{definition}
In \cite{cpp}, authors introduced the counterpart of the nearly spherical shells in the frameworks of domains with holes.
\begin{definition}[nearly annular sets]
Let $\Omega_0\setminus\overline{\Theta}\subset\R^d$ an open set where $\Omega_0$ and $\Theta$ are open, $\Theta\subset\subset\Omega_0$ and $0\in \Theta$. We say that $\Omega_0\setminus\overline{\Theta}$ is a \emph{$(R_1,R_2)-$nearly annular set}  if there exist $0<R_1<R_2<+\infty$ and $u,v\in W^{1,\infty}(\mathbb{S}^{d-1})$, with $||u||_{W^{1,\infty}}<R_2/2$, $||v||_{W^{1,\infty}}\le R_1/2$, such that 
\begin{equation*}
\partial\Omega_0= \left\{y \in \R^d \colon y=\xi  (R_2+u(\xi)), \, \xi \in \mathbb{S}^{d-1}\right\}
\end{equation*}
and
\begin{equation*}
\partial\Theta= \left\{y \in \R^d \colon y=\xi  (R_1+v(\xi)), \, \xi \in \mathbb{S}^{d-1}\right\}.
\end{equation*}
\end{definition}
The key points of the quantitative enhancements proved in this paper are the estimates of the deviations of an admissible convex set from the ball having the same perimeter and of a convex domain with smooth open hole from the spherical shell having the same volume and the same outer perimeter. So, we recall the following Hausdorff asymmetry indices, very useful in the framework of convex sets:
\begin{equation}\label{eq:ahstar}
 \mathcal{A}_{\mathcal{H}}^\star(E)= \min_{x \in \R^{d}} \left\{d_{\mathcal{H}}(E, B_r(x)) ,\, P(E)= P(B_r(x)) \right\},
\end{equation}
\begin{equation*}
    \mathcal{A}_{\mathcal{H}}^\sharp(E)=\min_{x \in \R^{d}} \left\{d_{\mathcal{H}}(E, B_r(x)) ,\, \abs{E}= \abs{B_r(x)} \right\}.
\end{equation*}
They give the deviations of $E$ from the balls having the same perimeter or the same volume as $E$ respectively.

$\mathcal{A}_{\mathcal{H}}^\star$ and $\mathcal{A}_{\mathcal{H}}^\sharp$ are linked by the following estimate. 
\begin{lemma}[\cite{gavitone2020quantitative}, Lemma 2.9]
    \label{gloria}
    Let $d \geq 2$, $\delta>0$ and let $E \subset \R^d$ be a bounded, convex, with $P(E)-P(E^\sharp) \leq \delta$  then 
    \begin{equation}
        \mathcal{A_H}^{\star}(E) \leq C(d)\mathcal{A_H}^{\sharp}(E).
    \end{equation}
\end{lemma}


With these definitions, we can recall the quantitative isoperimetric inequality proved in \cite{fuglede1989stability,fusco2015quantitative}. 

\begin{theorem}[Fuglede]
\label{fuglede}
    Let $d \geq 2$, and let $E$ be a bounded open and convex set with $ \abs{E}= \omega_d$. There exists $\delta, C$, depending only on $n$, such that if $P(E)-P(E^\sharp)\leq \delta$ then
    \begin{equation}
    \label{eq_Fuglede}
        P(E)-P(E^\sharp) \geq C g \bigl( \mathcal{A}_{\mathcal{H}}^{\sharp}(E) \bigr),
    \end{equation}
    where $g$ is defined in \eqref{function_g} by
    \begin{equation*}
        g(s)=
        \begin{cases}
            s^2 & \text{ if } d=2\\
            f^{-1}(s^2) & \text{ if } d=3\\
            s^{\frac{d+1}{2}} & \text{ if } d\geq 4
        \end{cases}
    \end{equation*}
    with $f(t)= \sqrt{t \log(\frac{1}{t})}$ for $0 < t < e^{-1}$.
\end{theorem}

We will make use of a modified version of the previous theorem, in terms of isovolumetric deficit $\abs{\Omega^{\star}} - \abs{\Omega}$. 

\begin{lemma}[\cite{amato2024estimates}, Lemma 2.4]
    \label{lemma_fugl_mod}
     Let $\Omega \subset \R^d$ be a bounded, open and convex set and let $\Omega^{\star}$ be the ball satisfying $P(\Omega)=P(\Omega^\star)=L$. Then, there exist $\delta,C$, depending only on $d$ and $L$, such that, if
    \begin{equation}
        \label{Hp_Lemma}
        \abs{\Omega^{\star}} - \abs{\Omega} \leq \delta
    \end{equation}
    then
    \[
     \abs{\Omega^{\star}} - \abs{\Omega}  \geq C g \bigl( \mathcal{A}_{\mathcal{H}}^{\star}(\Omega) \bigr)
    \]
    where $g$ is the function defined in \eqref{function_g}.
\end{lemma}

In the following section, we introduce another asymmetry index that is very helpful when dealing with convex sets.

\subsection{Some useful tools involving quermassintegrals}
Let us recall some basic facts about convex sets. Let $K \subset \R^d$ be a non-empty, bounded, convex set, let $B$ be the unitary ball centered at the origin and $\rho > 0$. We can write the Steiner formula for the Minkowski sum $K+ \rho B$ as
\begin{equation*}
    \abs{K + \rho B} = \sum_{i=0}^d \binom{d}{i} W_i(K) \rho^{i} .
\end{equation*}
We usually refer to the coefficients $W_i(K)$ as \emph{quermassintegrals} of $K$. Some remarkable cases are
$$W_0(K) = \abs{K},\quad W_1(K) =\frac{P(K)}{d},  \quad W_d(K) = \omega_d.$$

Quermassintegrals play a fundamental role when considering shape optimization problems on convex sets, as they provide powerful tools to handle convex super or sublevel sets. To this aim, we consider the Aleksandrov-Fenchel inequalities 
\begin{equation}
    \label{Aleksandrov_Fenchel_inequalities}
    \biggl( \frac{W_j(K)}{\omega_d} \biggr)^{\frac{1}{d-j}} \geq \biggl( \frac{W_i(K)}{\omega_d} \biggr)^{\frac{1}{d-i}} \qquad 0 \leq i < j \leq d-1,
\end{equation}
where equality holds if and only if $K$ is a ball. When $i=0$ and $j=1$, formula \eqref{Aleksandrov_Fenchel_inequalities} reduces to the classical isoperimetric inequality, i.e.
\[
P(K) \geq d \omega_d^{\frac{1}{d}} \abs{K}^{\frac{d-1}{d}}.
\]
An important role is played by 
\eqref{Aleksandrov_Fenchel_inequalities} with $i=1$ and $j=2$, that is
\begin{equation}
    \label{Aleksandrov_Fenchel_W_2}
    W_2(K) \geq d^{-\frac{d-2}{d-1}} \omega_d^{\frac{1}{d-1}} P(K)^{\frac{d-2}{d-1}}.
\end{equation}

Several improvements of inequalities involving the quermassintegrals are available in literature, see e.g. \cite{lasan} for recent results involving also the boundary momenta. In \cite{grosch} a quantitative version of the Alexandrov-Fenchel inequalities is proved, and we will report it  for  $j=2$ and $i=1$. In order to do that we have to define some objects. 

Let $\Omega\subset\R^d$ be an open convex set. We define \emph{support} function of $\Omega$ the map
\begin{equation*}
        h(\Omega, u):=\max_{x\in \Omega} (x\cdot u), \qquad u\in \mathbb{S}^{d-1}
    \end{equation*}
and \emph{width function} of $\Omega$
the map
$$w(\Omega, u):= h(\Omega, u) + h(\Omega, -u) \quad \text{for } u\in \mathbb{S}^{d-1}.$$
More precisely, the width $w(\Omega, u)$ is the distance between
the two support hyperplanes of $\Omega$ orthogonal to $u$. Roughly speaking, $w(\Omega, u)$ gives the thickness of $\Omega$ in the direction $u$. We point out that the diameter of $\Omega$ can be given in terms of $w(\Omega, \cdot)$ as follows:
$$D(\Omega)=\max_{u\in \mathbb{S}^{d-1}} w(\Omega, u).$$
The mean value of the width function is called the \emph{mean width} and it is denoted by 
$$w(\Omega):= \frac{2}{d\omega_d} \int_{\mathbb{S}^{d-1}}h(\Omega, u)\, du.$$
Finally, the Steiner point of $\Omega$ $s(\Omega)\in\Omega$ is defined as
    \begin{equation*}
        s(\Omega)=\frac{1}{\omega_d}\int_{\mathbb{S}^{d-1}}h(\Omega,u) u du.
    \end{equation*}

\begin{definition}
    Let $\Omega\subset\R^d$ be a bounded open convex set. The Steiner ball $B_{\Omega}$  of the convex set $\Omega$ is the ball centered at the Steiner point $s(\Omega)$ with diameter equal to $w(\Omega)$
\end{definition}
In Section \ref{main_sec_buco} we deal with domains with holes of the form $\Omega_0\setminus\overline{\Theta}$. In order to quantify the deviation of the convex set $\Omega_0$ from being a ball we make use of a quantitative version of \eqref{Aleksandrov_Fenchel_inequalities} in terms of the Hausdorff distance between $\Omega_0$ and its Steiner ball $B_{\Omega_0}$ (see \cite{grosch} for the proof).
\begin{theorem}[Quantitative Alexandrov-Fenchel inequalities] 
\label{quantitative_AF_propostion}
Let $\Omega_0$ be a convex body and $B_{\Omega_0}$ the Steiner ball of $\Omega_0$. Then,
\begin{equation}\label{stability_a_f}
\dfrac{W_2(\Omega_0)^{d-1}-\omega_d W_1(\Omega_0)^{d-2}}{W_1(\Omega_0)^{d-2}}\geq c(d,\Omega_0) d_{\mathcal{H}}(\Omega_0, B_{\Omega_0})^{\frac{d+3}{2}}.
\end{equation}  
\end{theorem}

From \eqref{stability_a_f}, it follows that 
\begin{equation}\label{eq:afq}
    W_2(\Omega_0)-\omega_d^{\frac{1}{d-1}}d^{-\frac{d-2}{d-1}}P(\Omega_0)^{\frac{d-2}{d-1}} \geq c d_{\mathcal{H}}(\Omega_0, B_{\Omega_0})^{\frac{d+3}{2(d-1)}}
\end{equation}

We will exploit the latter improved Aleksandrov-Fenchel inequality to estimate from below $\tilde{\sigma}_{p,q}(A_{R_1,R_2})-\tilde{\sigma}_{p,q}(\Omega_0\setminus\overline{\Theta})$ in terms of the asymmetry of convex set $\Omega_0$. Our technique is based on the application of \eqref{eq:afq} to the sublevel sets $\Omega_{0,t}$ for the test function $w_{\Omega_0}$ (see \eqref{GFerone}) which only depends on the distance from the boundary of $\Omega_0$. Anyway, our goal is to obtain the lower estimate not in terms of the asymmetry of $\Omega_{0,t}$, but in terms of the asymmetry of $\Omega_0$ itself. 
 For that reason, we need to control uniformly from below the asymmetry of the level subsets with the asymmetry of $\Omega_0$.  This technique is usually called \emph{propagation of the asymmetry}; it has been carried out in many situations (see also \cite{hansen1994isoperimetric} and \cite{bdph} for the results inspiring Lemma \ref{propagation_lemma}, involving the Fraenkel asymmetry), even in noneuclidean setting (see, for instance, \cite{cclp} or \cite{cclp2} for two applications in the Gauss space).

In our framework, we make use of the following result proved in \cite{massal}.
\begin{lemma}[Propagation of the Hausdorff asymmetry]
\label{propagation_lemma}
    Let $\Omega_0\subseteq \mathbb{R}^d$ be a bounded and convex set with positive measure and let $U\subset\Omega_0$ be with $|U|>0$ and such that 
    \begin{equation}\label{prop_hyp}
       d_{\mathcal{H}}(\Omega_0, U)\leq \dfrac{1}{2(d+2)}   d_{\mathcal{H}}(\Omega_0, B_{\Omega_0}), 
    \end{equation}
    where $B_{\Omega_0}$ is the Steiner ball of $\Omega_0$. Then, we have
    \begin{equation}\label{prop_thesis}
        d_{\mathcal{H}}(U, B_U)\geq \dfrac{1}{2}   d_{\mathcal{H}}(\Omega_0, B_{\Omega_0}).
    \end{equation}
\end{lemma}

\subsection{Some results on inner parallel sets and web functions}

We now recall some basic tools for handling a \emph{web function}, an evocative term coined by F. Gazzola (see \cite{gazzola1999existence}) for functions whose values are determined solely by the distance from the domain boundary. The name brings to mind the structure of a spider's web, an analogy drawn from the appearance of their level lines. We first introduce the notion of inner and outer parallel set.

\begin{definition}
Let $\Omega \subset \R^d$ be a convex set and $\rho(\Omega)$ be its inradius. For any $t \in [0,\rho(\Omega)]$, we define the \emph{inner parallel set} at distance $t$ relative to $\Omega$ the set
\[
\Omega_t := \Set{ x \in \Omega : d(x)>t } 
\]
where $d(x)$ is the distance of $x \in \Omega$ from the boundary of $\Omega$. For any $s\ge 0$, we define the \emph{outer parallel set} at distance $s$ relative to $\Omega$ the set
\[
\Omega^s := \Set{ x \in \R^d : d(x,\Omega)<s }. 
\]

\end{definition}

For the inner parallel sets the following differential estimate holds.

\begin{lemma}[\cite{brandolini2010upper}, Lemma 3.1]
    \label{lemma_derivata_perimetro_1}
    Let $\Omega$ be a bounded, convex, open set in $\R^d$. Then for almost every $t \in (0,r_{\Omega})$
    \begin{equation}
    \label{eq_lemma_1}
        -\frac{d}{dt} P(\Omega_t) \geq d(d-1) W_2(\Omega_t)
    \end{equation}
    and the equality holds if $\Omega$ is a ball.
\end{lemma}
A standard combination of the previous lemma with the chain rule, together with the equality $\abs{\nabla d(x)} = 1$ almost everywhere, entails the following result (see also \cite[Lemma 3.2]{brandolini2010upper} and \cite[Lemma 2.8]{paoli2020sharp}).
\begin{lemma}
    \label{lemma_derivata_perimetro_2}
    Let $f \colon [0,+\infty) \to [0,+\infty)$ be a nondecreasing $C^1$ function. 
    Set $u(x)=f(d(x))$ and
    \[
    E_t = \Set{ x \in \Omega \, : \, u(x) > t }
    \]
    then
    \begin{equation}
        \label{eq_lemma_2}
        -\frac{d}{dt} P(E_t) \geq d(d-1) \frac{W_2(E_t)}{\abs{\nabla u}_{u=t}}
    \end{equation}
\end{lemma}

A consequence of the previous lemma is that, if $f$ is \emph{nonincreasing} and we consider its \emph{sublevel} sets
 \[
    \hat{E}_t = \Set{ x \in \Omega \, : \, u(x) < t }
    \]
    then
    \begin{equation}
        \label{eq_lemma_2_NONINC}
        \frac{d}{dt} P(\hat{E}_t) \geq d(d-1) \frac{W_2(\hat{E}_t)}{\abs{\nabla u}_{u=t}}.
    \end{equation}

\section{The hybrid asymmetry for domains with holes}\label{sec:hybrid}

In this section we mention some basic facts about the asymmetry functional for domains with holes. We first explain our starting, intuitive idea to estimate the deviation of the hole, introduced by the current author, G. Paoli and G. Piscitelli in \cite{cpp}. Throughout the section we will always consider $p>1$ and $1\le q\le p$ in order to have the radiality of the minimizers for \eqref{pqbuco_min} for any spherical shell.

Let $z$ be the positive minimizer for $\tilde{\sigma}_{p,q}(A_{R_1,R_2})$ with $\|z\|_{L^q(\partial B_{R_2})}=1$ and denote by
$$z_m:=\inf_{A_{R_1,R_2}}z=z|_{\partial B_{R_1}}, \quad z_M:=\sup_{A_{R_1,R_2}}z=z|_{\partial B_{R_2}}.$$ 
Let define in $B_{R_2}$ the function
\begin{equation}\label{def_z}
\overline{z}:=\begin{cases}
    z &\text{in $A_{R_1,R_2}$},\\
    z_m &\text{in $\overline{B_{R_1}}$};
\end{cases}
\end{equation}
in other words, we extend the radial function $z$ in $\overline{B_{R_1}}$ with constant value $z_m$.

If $\Omega_0=B_{R_2}$ and $\Theta\neq B_{R_1}$, the most natural way to estimate the gap
$\tilde\sigma_{p,q}(A_{R_1,R_2})-\tilde\sigma_{p,q}(\Omega_0\setminus\overline{\Theta})=\tilde\sigma_{p,q}(B_{R_2}\setminus\overline{B_{R_1}})-\tilde\sigma_{p,q}(B_{R_2}\setminus\overline{\Theta})$ in terms of the deviation of $\Omega_0\setminus\overline{\Theta}=B_{R_2}\setminus\overline{\Theta}$ from $A_{R_1,R_2}=B_{R_2}\setminus\overline{B_{R_1}}$ is to estimate the deviation between the holes $\Theta$ and $B_{R_1}$. To do so, one can test both the variational formulations of $\tilde\sigma_{p,q}(B_{R_2}\setminus\overline{\Theta})$ and $\tilde\sigma_{p,q}(A_{R_1,R_2})$ with the same test function $\overline{z}$ (restricted respectively to each set). Let us notice that, in general, the restriction $\overline{z}|_{B_{R_2}\setminus\overline{\Theta}}$ is a test function for $B_{R_2}\setminus\overline{\Theta}$, but it is not a minimizer for the variational characterization in \eqref{pqbuco_min} if $\Theta\neq B_{R_1}$; on the other hand, $\overline{z}|_{A_{R_1,R_2}}=z$ attains the minimum $\tilde\sigma_{p,q}(A_{R_1,R_2})$. Moreover, it is immediate to see that $|B_{R_1}\setminus\overline{\Theta}|=|\Theta\setminus\overline{B_{R_1}}|$ (since $|B_{R_1}|=|\Theta|$ due to the volume constraint), that the boundary integrals coincide (since $\|\overline{z}\|_{L^q(\partial B_{R_2})}=\|z\|_{L^q(\partial B_{R_2})}=1$) and that $\nabla\overline{z}=0$ on $B_{R_1}$. 
Putting together all these remarks we get
\allowdisplaybreaks
\begin{align*}
\tilde{\sigma}_{p,q}(A_{R_1,R_2})&=\left(\int_{A_{R_1,R_2}}|\nabla z|^p\:dx+\int_{A_{R_1,R_2}}z^p\:dx\right)^{1/p}\\
&=\left(\int_{B_{R_2}}|\nabla \overline{z}|^p\:dx+\int_{B_{R_2}}\overline{z}^p\:dx-\int_{B_{R_1}}\overline{z}^p\:dx\right)^{1/p}\\
&=\left(\int_{B_{R_2}\setminus\overline{\Theta}}|\nabla\overline{z}|^p\:dx+\int_{\Theta\setminus\overline{B_{R_1}}}|\nabla \overline{z}|^p\:dx+\int_{B_{R_2}\setminus\overline{\Theta}}\overline{z}^p\:dx+\int_{\Theta}\overline{z}^p\:dx-\int_{B_{R_1}}\overline{z}^p\:dx\right)^{1/p}\\
&=\left(\int_{B_{R_2}\setminus\overline{\Theta}}|\nabla\overline{z}|^p\:dx+\int_{B_{R_2}\setminus\overline{\Theta}}\overline{z}^p\:dx+\int_{\Theta\setminus\overline{B_{R_1}}}|\nabla \overline{z}|^p\:dx+\int_{\Theta\setminus\overline{B_{R_1}}}\overline{z}^p\:dx-\int_{B_{R_1}\setminus\overline{\Theta}}\overline{z}^p\:dx\right)^{1/p}\\
&\ge\left(\tilde\sigma_{p,q}^p(B_{R_2}\setminus\overline{\Theta})+\int_{\Theta\setminus\overline{B_{R_1}}}|\nabla \overline{z}|^p\:dx+\int_{\Theta\setminus\overline{B_{R_1}}}\overline{z}^p\:dx-z_m^p|B_{R_1}\setminus\overline{\Theta}|\right)^{1/p}\\
&\ge\tilde\sigma_{p,q}(B_{R_2}\setminus\overline{\Theta})+\frac{ P(B_{R_2})^{\frac{p-1}{q}}}{p|A_{R_1,R_2}|^{\frac{p-1}{p}}}\int_{\Theta\setminus\overline{B_{R_1}}}\left(|\nabla \overline{z}|^p+\overline{z}^p-z_m^p\right)\:dx,
\end{align*}
where in the last inequality we used the Bernoulli inequality and the second geometric bound in Remark \ref{trivial} (see also the proof of Proposition \ref{teo:buco3} to see how the positive constant in front of the integral pops up in the general case). Since the last integral is strictly positive if and only if $\Theta\neq B_{R_1}$, it gives a quantitative estimate of the deviation of $\Theta$ from $B_{R_2}$. Of course, since the outer convex set is $\Omega_0=B_{R_2}$, the deviation of $B_{R_2}\setminus\overline{\Theta}$ from $A_{R_1,R_2}$ depends only on the deviation of $\Theta$ from $B_{R_1}$, as expected. The technical step forward is to reply the same idea with a general outer convex set $\Omega_0$.

In order to do that, we build a suitable test function for $\Omega_0\setminus\overline{\Theta}$ related to the normalized minimizer $z$ for the spherical shell $A_{R_1,R_2}$. In view of Proposition \ref{pro:radbuco}, $z$ is a multiple of $z_p$, minimizer for $\sigma_p(A_{R_1,R_2})$, thus it is radial. Since $z$ is radial on $A_{R_1,R_2}$, there exists $\tilde{\Psi}:[R_1,R_2]\to\R$ such that $z(y)=\tilde{\Psi}(|y|)$ for any $y\in A_{R_1,R_2}$. Moreover, in view of Proposition \ref{pro:prop_eigenf_buco}(iii), $\tilde{\Psi}'(r)>0$ for any $r\in]R_1,R_2[$.

We now introduce the function
\begin{equation}\label{eq:web}
{w}_{\Omega_0}(x):=
\begin{cases}
\tilde{\Psi}(R_2-d(x))\quad &\text{if} \  d(x)< R_2-R_1\\
z_m &\text{if} \ d(x)\geq R_2-R_1
\end{cases}
\quad x\in\Omega_0,
\end{equation}
 where $d(x)=d(x;\partial\Omega_0)$. To be clearer, the function $w_{\Omega_0}$ is built in the following way. Let $x\in\Omega_0$ with $d(x)=\delta\ge 0$; $w_{\Omega_0}(x)$ attains the same value as $\overline{z}$ at distance $\delta$ from $\partial B_{R_2}$. The map $w_{\Omega_0}$ depends only on the distance $d(x)$ from the outer boundary $\partial\Omega_0$ and is constant on the boundaries of the inner parallels to $\Omega_0$; in other words, $w_{\Omega_0}$ is a web function in $\Omega_0$.

An equivalent way to define the function $w_{\Omega_0}$ is the following. We first introduce  $\ell(t):=\abs{\nabla \overline{z}}_{\overline{z}=t}=\abs{\nabla z}_{z=t}$. The function $\ell(t)$ is well defined in $]z_m,z_M[$ in view of the radiality of $z$; moreover, $\ell(t)>0$ for any $t\in]z_m,z_M[$ since $\tilde{\Psi}'>0$. We now claim that
\begin{equation}\label{GFerone}
{w}_{\Omega_0}(x)=
\begin{cases}
\tilde{G}(d(x))\quad &\text{if} \  d(x)< R_2-R_1\\
z_m &\text{if} \ d(x)\geq R_2-R_1
\end{cases}
\quad x\in \Omega_0,
\end{equation}
where $\tilde{G}$ is defined by
\begin{equation*}
\tilde{G}^{-1}(t)=\int_{t}^{z_M}\dfrac{1}{\ell(\tau)}\;d\tau,
\end{equation*}
for any $t\in]z_m,z_M[$. Notice that $\tilde{G}^{-1}$ is strictly decreasing in $[z_m,z_M]$, then invertible, since the integrand function $1/\ell(t)$ is strictly positive; as a consequence $\tilde{G}$ is strictly decreasing in $[0,R_2-R_1]$.

 The equivalence between \eqref{eq:web} and \eqref{GFerone} is not immediately clear, so we prove it for the convenience of the reader. If $d(x)\ge R_2-R_1$, there is nothing to prove. If $d(x)\le R_2-R_1$, we show that for any $x\in\Omega_0$ such that $d(x)\le R_2-R_1$, it holds
\begin{equation}\label{eq:webequi}
    \tilde{G}(d(x))=\tilde{\Psi}(R_2-d(x)).
\end{equation}
To ease the notation, we set $r_x:=R_2-d(x)\in[R_1,R_2]$. Since $\tilde{\Psi}$ is invertible in $[R_1,R_2]$ with $\tilde{\Psi}(R_1)=z_m$ and $\tilde{\Psi}(R_2)=z_M$, one has
\begin{align*}
d(x)&=R_2-r_x=\tilde{\Psi}^{-1}(z_M)-\tilde{\Psi}^{-1}(\tilde{\Psi}(r_x))=\int_{\tilde{\Psi}(r_x)}^{z_M}\left(\tilde{\Psi}^{-1}\right)'(\tau)\:d\tau=\int_{\tilde{\Psi}(r_x)}^{z_M}\left(\left.\frac{1}{\tilde{\Psi}'(s)}\right|_{\tilde{\Psi}(s)=\tau}\right)\:d\tau\\
&=\int_{\tilde{\Psi}(r_x)}^{z_M}\dfrac{1}{\abs{\nabla z}_{z=\tau}}\;d\tau=\tilde{G}^{-1}(\tilde{\Psi}(r_x))=\tilde{G}^{-1}(\tilde{\Psi}(R_2-d(x))).
\end{align*}
By applying $\tilde{G}$ to the first and the last side of the equality we get \eqref{eq:webequi}.

We point out that in several references containing  comparison results via the web function method (see, for instance,  \cite{amato2024estimates,brandolini2010upper,bfnt,cpp}) the definition of the $w_{\Omega_0}$ is directly given via \eqref{GFerone}, since it turns out to be very useful in the computations for getting the required upper or lower estimates.

We conclude this section recalling that
\begin{gather*}
(w_{\Omega_0})_m:=\inf_{\Omega_0} w_{\Omega_0}\geq z_m,\\
(w_{\Omega_0})_M:=\sup_{\Omega_0} w_{\Omega_0} = z_M=\tilde{G}(0).
\end{gather*}
In addition, for every $t\in](w_{\Omega_0})_m,(w_{\Omega_0})_M[$ the following property holds:
$$|\nabla w_{\Omega_0}|_{w_{\Omega_0}=t}=|\nabla z|_{z=t}.$$
Indeed, recalling that $|\nabla d(x)|=1$, one has
\begin{align*}
|\nabla w_{\Omega_0}|_{w_{\Omega_0}=t}= |\tilde{G}'(d(x))|_{\tilde{G}(d(x))=t}|\nabla d(x)|=\frac{1}{|(\tilde{G}^{-1})'(t)|}=\frac{1}{\left|-\frac{1}{|\nabla z|_{z=t}}\right|}=|\nabla z|_{z=t}.
\end{align*}
We finally point out that $w_{\Omega_0}$ can be used as a test function for \emph{every} admissible holed set with outer domain $\Omega_0$.

\medskip

A crucial idea in our approach is the introduction of the so-called \emph{"shellifying"} hole $K_{\Omega_0}$, which plays the same role for $\Omega_0$ as the concentric spherical hole $B_{R_2}$ does when the exterior domain is the ball $B_{R_1}$.

\begin{definition}\label{def:k0}
Under the hypotheses and the notation above, we denote by $K_{\Omega_0}$ the inner parallel set relative to $\Omega_0$ having the same volume of $\Theta$, i.e. the set 
\begin{equation*}
K_{\Omega_0}
=\{ x\in\Omega_0\;:\; d(x)>t_{\Omega_0}  \},
\end{equation*}
with $t_{\Omega_0}\in(0,\rho (\Omega_0)) $ such that $|K_{\Omega_0}
|=|\Theta|$.
\end{definition}

The shellifying hole $K_{\Omega_0}$ enjoys the property that all points of $\partial K_0$ lie at the distance $t_{\Omega_0}$ from $\partial \Omega_0$. In this sense, the distance between the inner boundary and the outer boundary is kept uniform, as in a spherical shell. For this reason, we roughly refer to $\Omega_0 \setminus \overline{K_{\Omega_0}}$ as the "$\Omega_0$-shell".

Now we can give the following definition of weak inner asymmetry of the hole.

\begin{definition}[weak inner Asymmetry]
Under the hypotheses and the notation above, we denote with the symbol $\tilde{\mathcal{A}}(\Theta; \Omega_0)$ and call weak weighted Fraenkel-type asymmetry of  $\Theta$ relative to the set $\Omega_0$, the nonnegative quantity 
\begin{equation}\label{eq:atilde}
\tilde{\mathcal{A}}(\Theta;\Omega_0):=\int_{\Theta\setminus \overline{K_{\Omega_0}} 
}(|\nabla w_{\Omega_0}|^p+w_{\Omega_0}^p(x)-z_m^p)\:dx.
\end{equation}
\end{definition}
Let us notice that $|\nabla w_{(\Omega_0)}|^p+w_{\Omega_0}^p-z_m^p=0$ in $K_{\Omega_0}\setminus\Theta$, so $\tilde{\mathcal{A}}(\Theta,{\Omega_0})$ can be also written as
\[
\tilde{\mathcal{A}}(\Theta;{\Omega_0})=\int_{\Theta\triangle K_{\Omega_0}}(|\nabla w_{(\Omega_0)}|^p+w_{\Omega_0}^p-z_m^p)dx=Vol_{(|\nabla w_{(\Omega_0)}|^p+w_{\Omega_0}^p-z_m^p)}(\Theta\triangle K_{\Omega_0}),
\]
hence the name "weighted Fraenkel-type" asymmetry. The reason why we refer to $\tilde{\mathcal{A}}(\Theta;{\Omega_0})$ as a "weak" asymmetry will be clarified in Remark \ref{renna2}, once the full asymmetry functional will be introduced.

\medskip

Now, to estimate the deviation of the container $\Omega_0$ from a ball, the situation is more involved. Referring to \cite{cpp}, the linear case $p=q=2$ allows a Fuglede type approach, due to the fact that the minimizing function is explicit in that case. This allows to prove that the distance between $\Omega_0$ and $B_{R_2}$ can be given in terms of their Hausdorff distance, via the modulus of continuity $g$ defined in \eqref{function_g}. In the general case addressed here, however, the implicit nature of the minimizer $z$ for $\tilde{\sigma}_{p,q}(A_{R_1,R_2})$ prevents such a direct approach. 

In order to ease the readability of the section, we introduce the convex sublevel sets
\begin{equation}\label{set_def}
   \Omega_{0,t}=\{x\in \Omega_0\; | \; w_{\Omega_0}(x)<t  \}, \quad A_{0,t}=\{x\in B_{R_2}\;|\;\overline{z}(x)<t\}.
\end{equation}
Notiche that for any $t\in[z_m,z_M]$, $A_{0,t}$ is a ball since $\overline{z}$ is radial. Moreover
$$P(\Omega_{0,z_M})=P(\Omega_{0})=P(B_{R_2})=P(A_{0,z_M}).$$
In particular, using a standard comparison argument, we will show that $P(A_{0,t})\ge P(\Omega_{0,t})$ for any $t\in[z_m,z_M]$ in the Proof of Main Theorem 3. 

In order to give a quantitative enhancement of \eqref{bucostab_ineq_intro1} and introduce an appropriate asymmetry for $\Omega_0$, we need an improved, quantitative, version of $P(A_{0,t})\ge P(\Omega_{0,t})$, at least for the sublevel sets that are close to the maximal value.
\begin{lemma}\label{pro:pro}
Let $\Omega_0$ be a convex domain. Then, there exists a constant $T(\Omega_0)\in]z_m,z_M[$ such that, for any $t\in[T(\Omega_0),z_M]$, it holds
\begin{equation}\label{eq:quantinner}
P(A_{0,t})\ge P(\Omega_{0,t})+C(d)\tilde{G}^{-1}(t) d_\mathcal{H}^{\frac{d+3}{2(d-1)}}(\Omega_0;B_{\Omega_0})
\end{equation}
\end{lemma}
\begin{proof}
For any  $z_m<t<z_M$, combining \eqref{eq_lemma_2_NONINC} and \eqref{eq:afq} for the sublevel set $\Omega_{0,t}$ we have 
\begin{equation}\label{eq:asinner1}
    \frac{d}{dt} P(\Omega_{0,t})\geq d(d-1)\dfrac{W_2(\Omega_{0,t})}{\ell(t)}\geq C(d) \dfrac{\left(P(\Omega_{0,t})\right)^{\frac{d-2}{d-1}}}{\ell(t)}+C(d)\dfrac{d_{\mathcal{H}}(\Omega_{0,t},B_{\Omega_{0,t}})^{\frac{d+3}{2(d-1)}}}{\ell(t)},
\end{equation}
where 
$$ C(d)=d(d-1)d^{-\frac{d-2}{d-1}}w_d^{\frac{1}{d-1}}.$$
Since $A_{0,t}$ is a ball, it holds
\begin{equation}\label{eq:asinner2}
    \frac{d}{dt} P(A_{0,t})= C(d) \dfrac{\left(P(A_{0,t})\right)^{\frac{d-2}{d-1}}}{\ell(t)},
\end{equation}

Let us consider $$T_1(\Omega_0):=\inf\left\{s\in[z_m,z_M]: d_{\mathcal{H}}(\Omega_{0,s},B_{\Omega_{0,s}})
\ge\frac{d_{\mathcal{H}}(\Omega_{0},B_{\Omega_{0}})
}{2}\ and\ \frac{\ell(z_M)}{2}\le\ell(s)\le2\ell(z_M)\right\}$$
(notice that $T_1(\Omega_0)\neq z_M$). Now, we divide the proof in two different cases.

\medskip

\noindent \textbf{Case} $\mathbf{d=2}$. In this case, it is enough to consider $T(\Omega_0):=T_1(\Omega_0)$. Indeed, \eqref{eq:asinner1} and \eqref{eq:asinner2} become, respectively
\begin{equation}\label{eq:asinner3}
    \frac{d}{dt} P(\Omega_{0,t})\geq C(d) \dfrac{1}{\ell(t)}+C(d)\dfrac{d_{\mathcal{H}}(\Omega_{0,t},B_{\Omega_{0,t}})^{\frac{d+3}{2(d-1)}}}{\ell(t)}
\end{equation}
and
\begin{equation}\label{eq:asinner4}
    \frac{d}{dt} P(A_{0,t})= C(d) \dfrac{1}{\ell(t)}.
\end{equation}
Subtracting \eqref{eq:asinner4} from \eqref{eq:asinner3} and integrating from $t$ to $z_M$ we get that, for any $t\in[z_m,z_M]$:
\begin{equation*}
\int_t^{z_M}\frac{d}{ds} \left[P(\Omega_{0,s})-P(A_{0,s})\right]\:ds\ge C(d)\int_t^{z_M}\dfrac{d_{\mathcal{H}}(\Omega_{0,s},B_{\Omega_{0,s}})^{\frac{d+3}{2(d-1)}}}{\ell(s)}\:ds
\end{equation*}
Noticing that
\begin{equation*}
\int_t^{z_M}\frac{d}{ds} \left[P(\Omega_{0,s})-P(A_{0,s})\right]\:ds=P(\Omega_0)-P(B_{R_2})-P(\Omega_{0,t})+P(A_{0,t})=P(A_{0,t})-P(\Omega_{0,t}),
\end{equation*}
we conclude
\begin{equation}\label{eq:asinner5}
P(A_{0,t})\ge P(\Omega_{0,t})+ C(d)\int_t^{z_M}\dfrac{d_{\mathcal{H}}(\Omega_{0,s},B_{\Omega_{0,s}})^{\frac{d+3}{2(d-1)}}}{\ell(s)}\:ds
\end{equation}
By definition, for any $s\in [T(\Omega_0),z_M]$, it holds
$$
d_{\mathcal{H}}(\Omega_{0,s},B_{\Omega_{0,s}})
\ge\frac{d_{\mathcal{H}}(\Omega_{0},B_{\Omega_{0}})
}{2}.
$$
Plugging the estimate in \eqref{eq:asinner5} we finally get for any $t\in [T(\Omega_0),z_M]$
\begin{align*}
    P(A_{0,t})\ge P(\Omega_{0,t})+C'(d)d_{\mathcal{H}}(\Omega_{0},B_{\Omega_{0}})^{\frac{d+3}{2(d-1)}}\int_t^{z_M}\dfrac{1}{\ell(s)}\:ds,
\end{align*}
where the last integral is $\tilde{G}^{-1}(t)$ and $C'$ takes into account all the previous constants.

\medskip
\noindent \textbf{Case} $\mathbf{d>2}$. In higher dimension, the situation is more involved, since putting together \eqref{eq:asinner1} and \eqref{eq:asinner2} we just have
$$
P(A_{0,t})\ge P(\Omega_{0,t})+ C(d)\int_t^{z_M}\dfrac{d_{\mathcal{H}}(\Omega_{0,s},B_{\Omega_{0,s}})^{\frac{d+3}{2(d-1)}}}{\ell(s)}\:ds+C(d)\int_t^{z_M}\dfrac{\left(P(\Omega_{0,t})\right)^{\frac{d-2}{d-1}}-\left(P(A_{0,s})\right)^{\frac{d-2}{d-1}}}{\ell(s)}\:ds,
$$
where the last summand is negative and so not helpful for our purposes. In other words, it does not seem available a comparison on the whole interval $[z_m,z_M]$. We adopt a slightly different strategy. Thus, for any $t\in[T_1(\Omega_0),z_M]$ \eqref{eq:asinner1} becomes
$$
\frac{d}{dt} P(\Omega_{0,t})\geq C(d) \dfrac{\left(P(\Omega_{0,t})\right)^{\frac{d-2}{d-1}}}{\ell(t)}+C'(d)\dfrac{d_{\mathcal{H}}(\Omega_{0},B_{\Omega_{0}})^{\frac{d+3}{2(d-1)}}}{2\ell(z_M)}.
$$
On the other hand, in $t=z_M$ one has
\begin{align*}
\left.\frac{d}{dt} P(\Omega_{0,t})\right|_{t=z_M}&\geq C(d) \dfrac{\left(P(\Omega_{0})\right)^{\frac{d-2}{d-1}}}{\ell(z_M)}+C'(d)\dfrac{d_{\mathcal{H}}(\Omega_{0},B_{\Omega_{0}})^{\frac{d+3}{2(d-1)}}}{2\ell(z_M)}\\
&=C(d) \dfrac{\left(P(B_{R_2})\right)^{\frac{d-2}{d-1}}}{\ell(z_M)}+C'(d)\dfrac{d_{\mathcal{H}}(\Omega_{0},B_{\Omega_{0}})^{\frac{d+3}{2(d-1)}}}{2\ell(z_M)}\\
&=\left.\frac{d}{dt} P(A_{0,t})\right|_{t=z_M}+C'(d)\dfrac{d_{\mathcal{H}}(\Omega_{0},B_{\Omega_{0}})^{\frac{d+3}{2(d-1)}}}{2\ell(z_M)},
\end{align*}
i.e. the final derivative of $P(\Omega_{0,t})$ is strictly larger than the final derivative of $P(A_{0,t})$. As a consequence, there exists $T_2(\Omega_0)\in[z_m,z_M[$ such that
$$
\frac{d}{dt} P(\Omega_{0,t})\ge \frac{d}{dt} P(A_{0,t})+C'(d)\dfrac{d_{\mathcal{H}}(\Omega_{0},B_{\Omega_{0}})^{\frac{d+3}{2(d-1)}}}{4\ell(z_M)}
$$
for any $t\in[T_2(\Omega_0),z_M]$. Let us set $T(\Omega_0):=\max\{T_1(\Omega_0),T_2(\Omega_0)\}$. By construction, the estimate above becomes
$$
\frac{d}{dt} P(\Omega_{0,t})\ge \frac{d}{dt} P(A_{0,t})+C'(d)\dfrac{d_{\mathcal{H}}(\Omega_{0},B_{\Omega_{0}})^{\frac{d+3}{2(d-1)}}}{8\ell(t)}.
$$
Integrating as in the 2 dimensional case the thesis is achieved.
\end{proof}

We point out that the quantity $\tilde{G}^{-1}(t)$ could also be replaced by $z_M-t$, with a small modification of the multiplicative constant, since the denominators $\ell(z_M)$ and $\ell(s)$ appearing in the additional term are comparable.

Some comments are in order. We introduced the "critical level" $T(\Omega_0)$ in order to handle basically three things. Firstly, it is necessary a range of levels in which Lemma \ref{propagation_lemma} holds. Even if it seems natural that the inner parallel of a given convex set $\Omega_0$ has larger (or equal) asymmetry than $\Omega_0$ (roughly speaking, the more you go inside, the more you shrink), it does not seem available a proof of this fact up to our knowledge. A second issue is the lack of a quantitative comparison between $P(A_{0,t})$ and $P(\Omega_{0,t})$ for $d>2$ holding for any $t$, not only near $z_M$. Finally, the uniform bound on $\ell(s)$ in $[T(\Omega_0),z_M]$ could be removed if, for instance, $z$ was convex.

Anyway, the additional term appearing in \eqref{eq:quantinner} is strictly positive whenever $\Omega_0\neq B_{R_2}$ and equals zero if and only if $\Omega_0=B_{R_2}$.

\begin{definition}[outer Asymmetry]
Under the hypotheses and the notation above, we denote with the symbol $\tilde{\alpha}_{out}(\Omega_0)$ and call outer asymmetry of $\Omega_0$, the nonnegative quantity 
\begin{equation}\label{eq:out}
\tilde{\alpha}_{out}(\Omega_0):=(z_M-T(\Omega_0))^3d_\mathcal{H}^{\frac{d+3}{2(d-1)}}(\Omega_0;B_{\Omega_0}).
\end{equation}
where $T(\Omega_0)\in[z_m,z_M[$ is defined in Lemma \ref{pro:pro}.
\end{definition}

Just a few comments about the definition above. Given that quantitative inequalities for domains with holes are still in an early stage of development, our primary goal was to provide a functional capable of capturing the distance from radial symmetry to obtain a quantitative estimate that was previously unavailable. While the resulting outer asymmetry functional may appear specifically tailored to the $p,q$-setting, it fully meets our objective of characterizing the distance of $\Omega_0$ from a ball. Within the actual scope of this research, we consider this a satisfactory achievement as it fulfills the requirement of vanishing if and only if the domain $\Omega_0$ is a ball.

From a technical point of view, the estimate of the deviation of $\Omega_0$ from the spherical symmetry is obtained mixing an argument of propagation of asymmetry (i.e. comparing the asymmetry of some sublevel sets of the function $w_{\Omega_0}$ with the asymmetry of $\Omega_0$) with classical local comparison theorems between solutions of ordinary differential equations. The level $T(\Omega_0)<z_M$ is the lowest level for which both those comparisons hold. The factor $(z_M-T(\Omega_0))^3$ pops out after the integration of a squared height, see the proof of Proposition \ref{pro:buco2}. We point out that the ball from which we quantify the spherical deviation of $\Omega_0$ is the Steiner ball. Moreover, the power of the Hausdorff distance is the same appearing in the quantitative Aleksandrov-Fenchel inequality \eqref{eq:afq}, which is exploited to give a quantitative enhancement of the comparison between the perimeters of the sublevel sets of $w_{\Omega_0}$ and the perimeters of the spherical sublevel sets of $\overline{z}$. 

We are now in a position to define our full asymmetry functional for domains with holes.

\begin{definition}[Hybrid asymmetry]
\label{inner_asymmetry_def}
We define the hybrid asymmetry of $\Omega_0\setminus\overline{\Theta}$ as
\begin{equation} \label{hybrid}
\alpha_{hyb}(\Omega_0\setminus\overline{\Theta}):=\max\left\{\tilde{\alpha}_{out}(\Omega_0),\tilde{\mathcal{A}}(\Theta;\Omega_0)\right\}.
\end{equation} 
\end{definition}

The definition of hybrid asymmetry above is motivated by the following, natural, question. If $\Omega_0$ is the container, what is the most suitable hole to compare $\Theta$ with? Intuitively, it is neither a ball $B_r$ such that $|B_r|=|\Theta|$ nor the cavity of the maximal spherical shell $A_{R_1,R_2}$ (i.e. $B_{R_1}$), but rather the inner parallel $K_{\Omega_0}$ with the same measure as $\Theta$. 

This answer is motivated by the fact that $\partial K_{\Omega_0}$ and $\partial\Omega_0$ lie at uniform distance $t_{\Omega_0}$ (see Definition \ref{def:k0}), recollecting the analogous property of the spherical shell $A_{R_1,R_2}$, for which the boundaries $\partial B_{R_1}$ and $\partial B_{R_2}$ lie at uniform distance $R_2-R_1$. It is worth noting, however, that $K_{\Omega_0}=B_{R_1}$ if and only if $\Omega_0 = B_{R_2}$.

The quantity $\tilde{\alpha}_{out}(\Omega_0)$ quantifies the deviation of the domain $\Omega_0$ from its Steiner ball with respect to the Hausdorff distance. The quantity $\tilde{\mathcal{A}}(\Theta;\Omega_0)$ gives information about how far is $\Theta$ from the optimal cavity relatively to $\Omega_0$, namely $K_{\Omega_0}$. 
Anyway, both terms are necessary to estimate a distance from the spherical shell. Indeed, in the following proposition, we show that \eqref{hybrid} actually quantifies the deviation of an admissible set from the optimal set $A_{R_1,R_2}$.

\begin{proposition}
\label{renna}
The functional $\alpha_{hyb}$ is actually an asymmetry index, i.e. $\alpha_{hyb}(\Omega_0\setminus\overline{\Theta})\ge 0$ for any admissible set $\Omega_0\setminus\overline{\Theta}$ with $P(\Omega_0)=P(B_{R_2})$ and $|\Omega_0\setminus\overline{\Theta}|=|A_{R_1,R_2}|$ and 
$$\alpha_{hyb}(\Omega_0\setminus\overline{\Theta})=0\qquad\Leftrightarrow\qquad\Omega_0\setminus\overline{\Theta}=A_{R_1,R_2}.$$
\end{proposition}
\begin{proof}
If $\Omega_0\setminus\overline{\Theta}=A_{R_1,R_2}$, then $w_{\Omega_0}=\overline{z}$, $K_{\Omega_0}=B_{R_1}$, $B_{\Omega_0}=B_{R_2}$ and then it is immediate that $\alpha_{hyb}(\Omega_0\setminus\overline{\Theta})=0$.

Now, let us assume that $\alpha_{hyb}(\Omega_0\setminus\overline{\Theta})=0$ and show that $\Omega_0=B_{R_2}$ and $\Theta=B_{R_1}$. If $\alpha_{hyb}(\Omega_0\setminus\overline{\Theta})=0$, then both $\tilde{\alpha}_{out}(\Omega_0)=0$ and $\tilde{\mathcal A}(\Theta;\Omega_0)=0$. Let us argue by contradiction and suppose that $\Omega_0\setminus\overline{\Theta}\neq A_{R_1,R_2}$. Since $\tilde{\alpha}_{out}(\Omega_0)=0$, then necessarily $\Omega_0$ coincides with its Steiner ball and has the same perimeter as $B_{R_2}$; consequently $\Omega_0$ is itself the ball $B_{R_2}$, $w_{\Omega_0}=\overline{z}$ and $K_{\Omega_0}=B_{R_1}$. Then, since we are under the hypothesis $\Omega_0\setminus\overline{\Theta}\neq A_{R_1,R_2}$, $\Theta$ cannot coincide with $B_{R_1}$. As a consequence we obtain
$$\tilde{\mathcal A}(\Theta;\Omega_0)=\tilde{\mathcal A}(\Theta;B_{R_2})=\int_{\Theta\setminus \overline{B_{R_1}} 
}(|\nabla \overline{z}|^p+\overline{z}^p(x)-z_m^p)\:dx>0,$$
since $\overline{z}>z_m$ in $\Theta\setminus \overline{B_{R_1}}$, getting a contradiction. 
\end{proof}

In the following remark, we clarify why we refer to $\tilde{\mathcal{A}}(\Theta;\Omega_0)$ as a "weak" asymmetry; this justifies the necessity of both terms for an asymmetry functional defined on domains with holes.

\begin{remark}\label{renna2}
We point out that $\tilde{\mathcal{A}}(\Theta;\Omega_0)$ could be zero for some $\Theta \neq K_{\Omega_0}$ (and in this sense is called "weak"). This happens whenever a hole $\Theta \neq K_{\Omega_0}$ lies at a distance larger than $R_2 - R_1$ from $\partial\Omega_0$, i.e., in the region of $\Omega_0$ where $w_{\Omega_0}$ takes the constant value $z_m$. More precisely, if$$t_{\Omega_0} > R_2 - R_1 \quad \text{and} \quad 0 < d_{\mathcal{H}}(\Theta, K_{\Omega_0}) < t_{\Omega_0} - (R_2 - R_1),$$
then $w_{\Omega_0} \equiv z_m$ on $\Theta \setminus K_{\Omega_0}$, and so $\tilde{\mathcal{A}}(\Theta;\Omega_0)=0$. In other words, the actual integration domain for $\tilde{\mathcal{A}}(\Theta;\Omega_0)$ is outside the plateau where $w_{\Omega_0} \equiv z_m$, see Figure \ref{fig:hybridvera}.

 \begin{figure}[!h]
    \centering    \includegraphics[width=.3\textwidth]{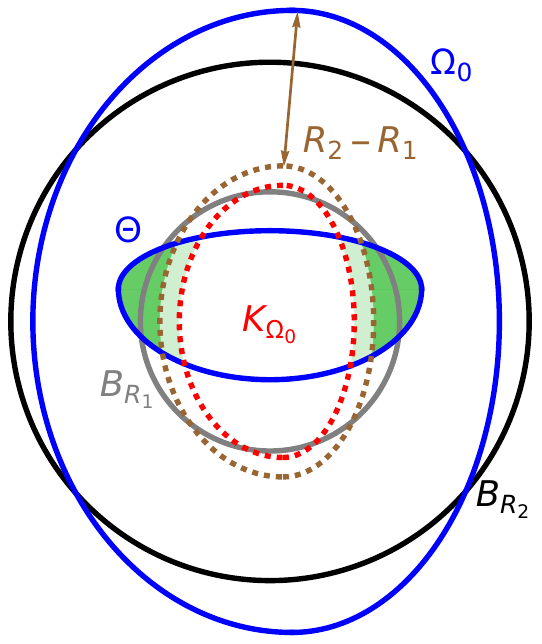}
\caption{The darker green set is the actual integration domain in the definition of $\tilde{\mathcal A}(\Theta;\Omega_0)$.}
    \label{fig:hybridvera}
\end{figure}
Nevertheless, in that case, the inequality $t_{\Omega_0} > R_2 - R_1$ implies that $\Omega_0$ cannot coincide with $B_{R_2}$; thus, the global asymmetry satisfies
\begin{equation*}
\alpha_{hyb}(\Omega_0 \setminus \overline{\Theta}) = \max\left\{ \tilde{\alpha}_{out}(\Omega_0), \tilde{\mathcal{A}}(\Theta;\Omega_0)\right\} = \tilde{\alpha}_{out}(\Omega_0) > 0.
\end{equation*}
This is also consistent with Proposition \ref{renna}, since $\Omega_0 \setminus \overline{\Theta} \neq A_{R_1,R_2}$ cannot have zero deviation from the spherical shell.
\end{remark}

We point out another interesting feature of the hybrid asymmetry. The functional $\alpha_{hyb}$ actually quantifies both the distance from being spherical of $\Omega_0$, and the global distance of $\Omega_0\setminus\overline{\Theta}$ from the "$\Omega_0$-shell" $\Omega_0\setminus\overline{K_{\Omega_0}}$. As a consequence, it detects the best competitor between, for instance, two domains having the same Hausdorff or Fraenkel distance from the spherical shell, but with different mass distribution around the hole. To put it another way, $\alpha_{hyb}$ is lower if $\Omega_0\setminus\overline{\Theta}$ is closer to $\Omega_0\setminus\overline{K_{\Omega_0}}$. For instance, $\alpha_{hyb}$ is useful to detect the most symmetric set in the case of two equal containers with the same hole, but in different relative position (see Figure \ref{simasim}).

\begin{figure}[!h]
\centering
\includegraphics[width=.3\textwidth]{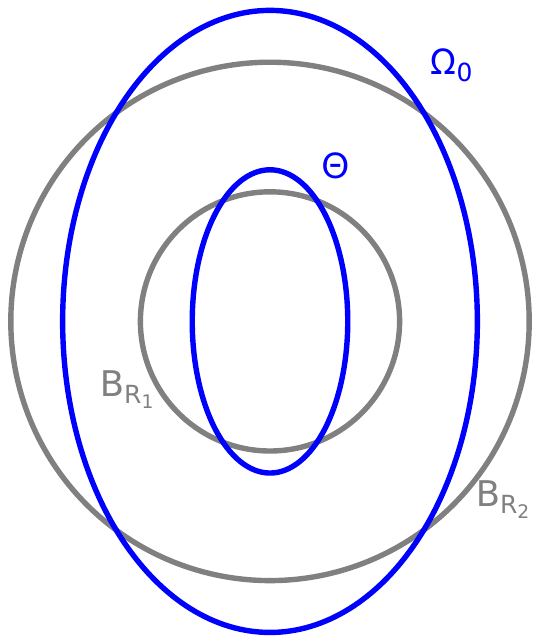}\qquad\includegraphics[width=.3\textwidth]{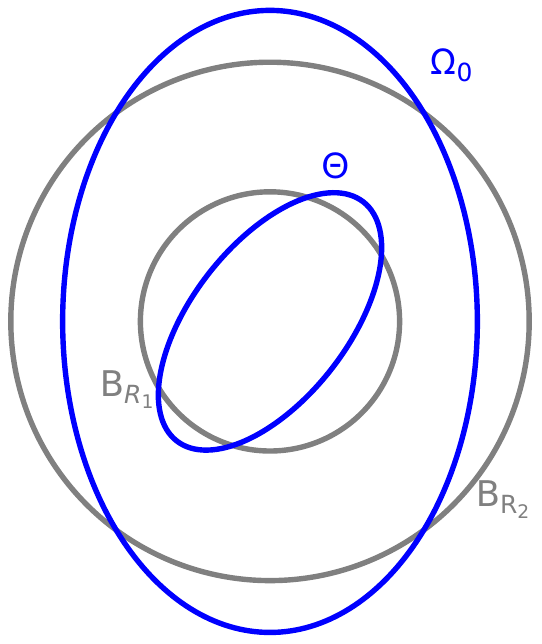}
\caption{The blue set on the left is clearly more symmetric than the blue set on the right. The hybrid asymmetry recognizes their difference, even though their canonical Hausdorff and Fraenkel distances from the gray spherical shell are the same.}
\label{simasim}
\end{figure}

\section{Proof of the Main Results: convex domains}
\label{main_sec_conv}
In this section we provide the proof of the main results for convex sets. The core of the section is the next proposition, that gives us the maximality of the ball for $\sigma_{p,q}$. The proof follows a similar scheme to Theorem 1.2 in \cite{amato2024estimates}.

\begin{proposition}\label{pro:gferone}
Let $\Omega\subset\R^d$ be convex and let $\Omega^\star$ be the ball having the same perimeter as $\Omega$. Then
$$\sigma_{p,q}(\Omega^\star)\ge\sigma_{p,q}(\Omega)+C(p,q,d,P(\Omega^\star))(|\Omega^\star|-|\Omega|).$$
\end{proposition}
\begin{proof}
In order to avoid the introduction of too many new symbols, in this proof we repeat some of the notation used in the previous sections. Let us denote by $z$ the minimizer for $\sigma_{p,q}(\Omega^\star)$ with $\|z\|_{L^q(\partial\Omega^\star)}=1$ and let $z_m$ and $z_M$ its infimum and supremum, respectively. Let us build a suitable positive test function $w\in W^{1,p}(\Omega)$ such that
\begin{equation}\label{wp} \int_{\Omega^\star}z^p\:dx\ge\int_{\Omega}w^p\:dx+z_m^p(|B_R|-|\Omega|),
\end{equation}
\begin{equation}\label{dwp} 
\int_{\Omega^\star}|\nabla z|^p\:dx\ge\int_{\Omega}|\nabla w|^p\:dx,
\end{equation}
\begin{equation}\label{wq} 
\int_{\partial \Omega^\star}z^q\:d\mathcal{H}^{d-1}=\int_{\partial\Omega}w^q\:d\mathcal{H}^{d-1}=1.
\end{equation}
As was anticipated, we will make use of the web functions method. Let us set $\ell(t):= \abs{\nabla z}_{z=t}$ for any $z_m <t<z_M$ and observe that $\ell(t)>0$ for any $z_m <t<z_M$ combining the results in Proposition \ref{pro:eigball} and Proposition \ref{pro:radball}. 
We define $w \in W^{1,p}(\Omega)$ in a similar way as done with the function $w_{\Omega_0}$ in Section \ref{sec:hybrid}, by setting $$w(x):= G(d(x)), \, x \in \Omega,$$ where $G:[0,\rho(\Omega)]\to[z_m,z_M]$ is the inverse of the function $$\displaystyle{G^{-1}(t) :=\int_t^{z_M} \frac{1}{\ell(\tau)} \, d\tau}\quad \text{with $z_m<t<z_M$}.$$ 
We point out that $G$ plays the same role as $\tilde{G}$ in Section \ref{sec:hybrid}, and $w(x)=z(y)$ if $d(x,\partial\Omega)=d(y,\partial\Omega^*)$. Moreover, the web function $w$ belongs to $W^{1,p}(\Omega)$, it is positive and the following statements hold:
\begin{equation}
        \begin{gathered}
        w_M:= \sup_\Omega w= z_M,\\
        w_m:=\inf_{\Omega} w \ge z_m,\\
        \abs{\nabla w}_{w=t}= \abs{\nabla z}_{z=t}=\ell(t)>0, \quad w_m<t<w_M.
        \end{gathered}
    \end{equation}
Notice that, as already seen in Section \ref{sec:hybrid}, due to the radiality of $z$, $w$ is equivalently defined via the function $\Psi:[0,R]\to\R$, such that $z(x)=\Psi(|x|)$; in particular $w(x)=\Psi(R-d(x))$ for every $x\in\Omega_0$.
    
To get \eqref{wp}, \eqref{dwp} and \eqref{wq}, we need to handle the sublevel sets of $z$ and $w$. We introduce
    \begin{equation*}
        \begin{aligned}
        \hat{E}_t &:= \{ x \in \Omega \,:\, {w} (x)<t\}, & \hat{B}_t &:= \{ x \in \Omega^\star \,:\, {z} (x)<t\} ,\\
        E_t&:= \{ x \in \Omega \,:\, w(x)>t\}= \Omega \setminus \overline{\hat{E}}_{t}, & B_t &:= \{ x \in \Omega^\star \,:\, z(x)>t\}=\Omega^\star \setminus \overline{\hat{B}}_{t}.
        \end{aligned}
    \end{equation*}
For the function $w=G(d(x))$ formula \eqref{eq_lemma_2_NONINC} holds; combining the latter with the Aleksandrov-Fenchel inequality \eqref{Aleksandrov_Fenchel_W_2} between $P(\cdot)$ and $W_2(\cdot)$ we have
    \[
    \frac{d}{dt}P(\hat{E}_t)\ge d(d-1) \frac{W_2(\hat{E}_t)}{\ell(t)}\ge d(d-1)d^{-\frac{d-2}{d-1}}\omega_d^{\frac{1}{d-1}}\frac{(P(\hat{E}_t))^{\frac{d-2}{d-1}}}{\ell(t)},
    \]
    while for the ball $\hat{B}_t$ equalities hold:
    \[
    \frac{d}{dt}P(\hat{B}_t)= d(d-1) \frac{W_2(\hat{B}_t)}{\ell(t)}= d(d-1)d^{-\frac{d-2}{d-1}}\omega_d^{\frac{1}{n-1}}\frac{(P(\hat{B}_t))^{\frac{d-2}{d-1}}}{\ell(t)}.
    \]
    Now, $P(\hat{E}_{z_M})=P(\Omega)=P(\Omega^\star)=P(B_{z_M})$. Then, by classical comparison theorems for differential inequalities we get
    \begin{equation}\label{compare_tilde}
        P(\hat{E}_t)\le P(\hat{B}_t) 
    \end{equation}
    a.e. in $[z_m,z_M]$. In order to have \eqref{wp}, we set 
    $$\hat{\mu}(t) := \lvert \hat{E}_t \rvert \quad,\quad  \hat{\nu}(t) := \lvert \hat{B}_t \rvert$$
    and
    $$\mu(t):=\abs{E_t}= \abs{\Omega}- \hat{\mu}(t)\quad,\quad \nu(t):=\abs{B_t}=\abs{\Omega^\star}- \hat{\nu}(t).$$
The previous distribution functions are absolutely continuous; in order to obtain a comparison in their intervals of definition, we apply the coarea formula \eqref{coarea} on the sublevel sets $\hat{E}_t$ and $\hat{B}_t$ as follows:
\begin{align*}
\hat{\mu}(t)&=\int_{\hat{E}_t}\frac{|\nabla{w}|}{|\nabla{w}|}\:dx=\int_{z_m}^{t}\left(\int_{{w}=s}\frac{1}{\abs{\nabla {w}}} \, d \mathcal{H}^{d-1}\right)ds,\\
\hat{\nu}(t)&=\int_{\hat{B}_t}\frac{|\nabla{z}|}{|\nabla{z}|}\:dx=\int_{z_m}^{t}\left(\int_{{z}=s}\frac{1}{\abs{\nabla {z}}} \, d \mathcal{H}^{d-1}\right)ds.
\end{align*}
We thus obtain
    \begin{align*}
    \hat{\mu}'(t)=&\int_{{w}=t}\frac{1}{\abs{\nabla {w}}} \, d \mathcal{H}^{d-1}= \frac{P(\hat{E}_t)}{\ell(t)}\le \frac{P(\hat{B}_t)}{\ell(t)}\\
        =&\int_{{z}=t}\frac{1}{\abs{\nabla {z}}} \, d \mathcal{H}^{d-1}=\hat{\nu}'(t)
    \end{align*}
 a.e. in $[z_m,z_M]$. Consequently, inequality $-\mu'(t)\leq -\nu'(t)$ holds a.e. in $[z_m,z_M]$. Integrating from $t$ to $z_M$, since $\nu(z_M)=\mu(z_M)=0$, we have
 $$\mu(t)-\nu(t)=\int_t^{z_M}(\nu'(t)-\mu'(t))\:dt\le 0,$$
 and so we conclude that $\mu(t)\le\nu(t)$ in $[z_m, z_M]$. We point out also that, by construction, $\mu(t)=|\Omega|$ and $\nu(t)=|\Omega^\star|$ for any $t\in[0,z_m]$. This yields the following estimate
    \begin{align*}
        \int_\Omega w^p \, dx=&\int_0^{w_M}p t^{p-1} \mu(t)\, dt =\int_0^{z_M}p t^{p-1} \mu(t)\, dt=\int_{0}^{z_m}p t^{p-1} \mu(t)\, dt+\int_{z_m}^{z_M}p t^{p-1} \mu(t)\, dt \\
&\le\int_{0}^{z_m}p t^{p-1} \mu(t)\, dt+\int_{z_m}^{z_M}p t^{p-1} \nu(t)\, dt=\int_{0}^{z_M}p t^{p-1} \nu(t)\, dt-\int_{0}^{z_m}p t^{p-1} (\nu(t)-\mu(t))\, dt\\
        =&\int_{\Omega^\star}z^p\, dx- z_m^p(\abs{\Omega^\star}-\abs{\Omega}),
    \end{align*}
    i.e. inequality \eqref{wp} is proved. 
    
    Furthermore, \eqref{compare_tilde} entails the following estimate for the integrals of the gradients on the level sets for a.e. $t\in[z_m,z_M]$:
    \begin{equation*}
        \int_{{w}=t} \abs{\nabla {w}}^{p-1} \, d\mathcal{H}^{d-1}=\ell(t)^{p-1}P(\hat{E}_t)\le \ell(t)^{p-1}P(\hat{B}_t)=\int_{{z}=t} \abs{\nabla {z}}^{p-1}\, d\mathcal{H}^{d-1}.
    \end{equation*}
    Using again the coarea formula we obtain \eqref{dwp}:
    \begin{align*}
        \int_\Omega \abs{\nabla w}^p \, dx=& \int_{w_m}^{w_M} \int_{{w}=t} \abs{\nabla {w}}^{p-1}\, d\mathcal{H}^{d-1} \, dt \leq \int_{z_m}^{z_M} \int_{{z}=t} \abs{\nabla{z}}^{p-1}\, d\mathcal{H}^{d-1} \, dt=\int_{\Omega^\star} \abs{\nabla z}^p.
    \end{align*}
    We finally observe that \eqref{wq}  is an immediate consequence of the definition of $w$, which is constant on $\partial \Omega$, getting the same value as $z|_{\partial\Omega^\star}=z_M$; so, in view of the equality $P(\Omega)=P(\Omega^\star)$:
    \begin{equation*}
        \int_{\partial\Omega} w^q \, d\mathcal{H}^{n-1}= w_M^qP(\Omega)=z_M^qP(\Omega^\star)=\int_{\partial\Omega^\star} z^q \, d\mathcal{H}^{n-1}=1.
    \end{equation*}
Plugging \eqref{wp}, \eqref{dwp} and \eqref{wq} in \eqref{pq_min} we obtain
\allowdisplaybreaks
\begin{align*}
\sigma_{p,q}(\Omega^\star)&=\dfrac{\left(\ds\int _{\Omega^\star}|\nabla z|^p\;dx+\ds\int_{\Omega^\star}z^p\;dx\right)^{1/p}}{\left(\ds\int_{\partial \Omega^\star}z^q \;d\mathcal{H}^{d-1}\right)^{1/q}}\ge\dfrac{\left(\ds\int _{\Omega}|\nabla w|^p\;dx+\ds\int_{\Omega}w^p\;dx+z_m^p(|\Omega^\star|-|\Omega|)\right)^{1/p}}{\left(\ds\int_{\partial \Omega}w^q \;d\mathcal{H}^{d-1}\right)^{1/q}}\\
&\ge \left(\sigma^p_{p,q}(\Omega)+z_m^p(|\Omega^\star|-|\Omega|)\right)^{1/p}=\sigma_{p,q}(\Omega)\left(1+\frac{z_m^p(|\Omega^\star|-|\Omega|)}{\sigma^p_{p,q}(\Omega)}\right)^{1/p}\\
&\ge\sigma_{p,q}(\Omega)\left(1+\frac{z_m^p(|\Omega^\star|-|\Omega|)}{p\cdot\sigma^p_{p,q}(\Omega)}\right)=\sigma_{p,q}(\Omega)+\frac{z_m^p(|\Omega^\star|-|\Omega|)}{p\cdot\sigma^{p-1}_{p,q}(\Omega)}\\
&\ge\sigma_{p,q}(\Omega)+\frac{z_m^p P(\Omega)^{\frac{p-1}{q}}}{p|\Omega|^{\frac{p-1}{p}}}(|\Omega^\star|-|\Omega|)\ge\sigma_{p,q}(\Omega)+\frac{z_m^p P(\Omega^\star)^{\frac{p-1}{q}}}{p|\Omega^\star|^{\frac{p-1}{p}}}(|\Omega^\star|-|\Omega|)\\
&=\sigma_{p,q}(\Omega)+\frac{z_m^p}{p}\left(\frac{ P(\Omega^\star)^{\frac{1}{q}-\frac{d}{p(d-1)}}}{d^{\frac{d}{d-1}}\omega_d^{\frac{1}{d-1}}}\right)^{p-1}(|\Omega^\star|-|\Omega|),
\end{align*}
where the inequality between the second and the third line is due to Bernoulli's inequality.
\end{proof}

\begin{remark}
For the nonlinear eigenvalue $\sigma_p^p(\cdot)$, Proposition \ref{pro:gferone} reduces to
    $$\sigma^p_{p}(\Omega^\star)\ge\sigma^p_{p}(\Omega)+z_m^p(|\Omega^\star|-|\Omega|).$$
\end{remark}

Now we are in position to prove both Main Theorem 1 and Main Theorem 2 stated in the introduction. 

\begin{proof}[Proof of Main Theorem 1]
It is a straightforward consequence of Proposition \ref{pro:gferone}, since $|\Omega^\star|\ge|\Omega|$ and all the equalities in the proof occur if and only if $\Omega=\Omega^\star$.
\end{proof}

\begin{proof}[Proof of Main Theorem 2]
It is a straightforward consequence of Proposition \ref{pro:gferone} and Lemma \ref{lemma_fugl_mod}.
\end{proof}

\begin{remark}
It is interesting to show that the stability result, as stated in terms of the Hausdorff asymmetry, does not hold if $\Omega$ is not nearly spherical. In other words, the gap $\sigma_{p,q}(\Omega^\star)-\sigma_{p,q}(\Omega)$ cannot be controlled from below by an arbitrary large Hausdorff asymmetry. We show this fact with an explicit example. Let us consider the case $d=4$, $p=q=2$. Let us assume that we can estimate with the function $g$ the gap $\sigma_{2,2}(\Omega^\star)-\sigma_{2,2}(\Omega)$ for any convex set in $\R^4$ regardless of its asymmetry, namely that one has
\begin{equation}\label{eq:contro4}
\sigma_{2,2}(\Omega^\star)-\sigma_{2,2}(\Omega)\ge C g\left(\mathcal{A}^\star_{\mathcal{H}}(\Omega)\right)=C \left(\mathcal{A}^\star_{\mathcal{H}}(\Omega)\right)^\frac52    
\end{equation}
(where $C>0$ only depends on $P(\Omega)=P(\Omega^\star)$) even for a set $\Omega$ with large Hausdorff asymmetry. Let us consider, for any $\varepsilon\in(0,1)$, the thinning cylinders
$$\Omega_\varepsilon=D_\varepsilon\times\left[-\frac{2\pi}{\varepsilon^2}+\frac{\varepsilon}{3},\frac{2\pi}{\varepsilon^2}-\frac{\varepsilon}{3}\right],$$
where $D_\varepsilon$ is the $3$-dimensional ball in $\{x_4=0\}$ with radius $\varepsilon$.
It holds
$$P(\Omega_\varepsilon)=4\pi\varepsilon^2\left(\frac{4\pi}{\varepsilon^2}-\frac{2\varepsilon}{3}\right)+2\cdot\frac{4}{3}\pi \varepsilon^3=16\pi^2,\quad |\Omega_\varepsilon|=\frac{4}{3}\pi \varepsilon^3 \left(\frac{4\pi}{\varepsilon^2}-\frac{2\varepsilon}{3}\right)=\frac{16}{3}\pi^2 \varepsilon-\frac{8}{9}\pi \varepsilon^4\to 0$$
as $\varepsilon\to 0$. So, $\Omega_\varepsilon^\star=B_2$ (since in $\R^4$ $P(B_R)=2\pi^2 R^3$) and we have
$$\mathcal{A}^\star_{\mathcal{H}}(\Omega_\varepsilon)\simeq \frac{2\pi}{\varepsilon^2}-2.$$
On the other hand, in view of Remark \ref{trivial}, it holds
$$\sigma_{2,2}(\Omega_\varepsilon)\le\frac{|\Omega_\varepsilon|^{1/2}}{P(\Omega_\varepsilon)^{1/2}}\to 0.$$
By \eqref{eq:contro4} we have
$$\sigma_{2,2}(B_2)-\sigma_{2,2}(\Omega_\varepsilon)\ge C\mathcal{A}^\star_{\mathcal{H}}(\Omega_\varepsilon)^\frac52\simeq \frac{C'}{\varepsilon^5},$$
thus, letting $\varepsilon\to0$, we get a contradiction.
\end{remark}

\section{Proof of the main results: domains with holes}\label{main_sec_buco}
In this section, we give the proof of the results for domains with holes. For the comfort of the reader, we will divide the proofs mainly in three parts: we first prove the optimality of $A_{R_1,R_2}$; then we provide quantitative enhancements of inequality \eqref{bucostab_ineq_intro1} in terms of the inner and of the outer asymmetry and finally we achieve the stability result.

For the maximality of the spherical shell (Main Theorem 3) we adapt the proof of Proposition \ref{pro:gferone}. We point out that the most relevant difference in this case is due to the cancellation of the term of deficit of measure, due the volume constraint $|\Omega_0\setminus\overline{\Theta}|=|A_{R_1,R_2}|$. 

\begin{proof}[Proof of Main Theorem 3] We follow the same scheme as in Proposition \ref{pro:gferone}. So, we omit the parts that are repeated verbatim. 

Let us consider the function $w_{\Omega_0}$ introduced in Section \ref{sec:hybrid} and recall that $w_{\Omega_0}\in W^{1,p}(\Omega_0)$, $w_{\Omega_0}(x)=z_m$ if $d(x)\ge R_2-R_1$  and
\begin{gather*}
|\nabla w_{\Omega_0}|_{w_{\Omega_0}=t}=|\nabla z|_{z=t},\\
(w_{\Omega_0})_m:=\min_{\Omega_0} w_{\Omega_0}\geq z_m,\\
(w_{\Omega_0})_M:=\max_{\Omega_0} w_{\Omega_0} = z_M=\tilde{G}(0).
\end{gather*}

We repeat verbatim the same arguments as in the proof of Proposition \ref{pro:gferone}, applying the comparison arguments based on \eqref{eq_lemma_2_NONINC} to the sublevel sets 
\begin{equation*}
   \Omega_{0,t}=\{x\in \Omega_0\; | \; w_{\Omega_0}(x)<t  \}, \quad A_{0,t}=\{x\in B_{R_2}\;|\;\overline{z}(x)<t\}
\end{equation*}
 defined in \eqref{set_def}. We obtain the comparison
$$P(A_{0,t})\ge P(\Omega_{0,t})$$
for a.e. $t\in]z_m,z_M[$. Notice that the argument works since both $w_{\Omega_0}$ and $\overline{z}$ are decreasing with respect to the distance from the outer boundaries ($\partial\Omega_0$ and $\partial B_{R_2}$ respectively) and the balls $A_{0,t}$ attain the equality in \eqref{eq_lemma_2_NONINC}. We thus proceed as in Proposition \ref{pro:gferone} and obtain the following estimates between the integrals involving $\overline{z}$ and $w_{\Omega_0}$ on $B_{R_2}$ and $\Omega_0$, respectively: 
\begin{equation}\label{zp}
\int_{B_{R_2}}\overline{z}^p\:dx\ge\int_{\Omega_0}w_{\Omega_0}^p\:dx+z_m^p(|B_{R_2}|-|\Omega_0|)
\end{equation}
\begin{equation}\label{dzp}
\int_{B_{R_2}}|\nabla\overline{z}|^p\:dx\ge\int_{\Omega_0}|\nabla w_{\Omega_0}|^p\:dx
\end{equation}
\begin{equation}\label{zq}
    \int_{\partial B_{R_2}}z^q\:d\mathcal{H}^{d-1}=\int_{\partial B_{R_2}}\overline{z}^q\:d\mathcal{H}^{d-1}=\int_{\partial\Omega_0}w_{\Omega_0}^q\:d\mathcal{H}^{d-1}
\end{equation}
Our aim is to get the following, similar, estimates for the volume integrals on the domains with holes (equality \eqref{zq} for the boundary terms can be directly plugged in the variational characterization of $\tilde{\sigma}_{p,q}$):
\begin{equation}\label{zpbuco}
\int_{A_{R_1,R_2}}z^p\:dx\ge\int_{\Omega}w_{\Omega_0}^p\:dx
\end{equation}
\begin{equation}\label{dzpbuco}
\int_{A_{R_1,R_2}}|\nabla z|^p\:dx\ge\int_{\Omega}|\nabla w_{\Omega_0}|^p\:dx
\end{equation}
    
To obtain \eqref{zpbuco}, we start splitting the integrals in \eqref{zp} separating the plain parts (respectively $\Omega_0\setminus\overline{\Theta}$ and $A_{R_1,R_2}$) and the holes (respectively $\Theta$ and $B_{R_1}$):
$$
\int_{A_{R_1,R_2}}\overline{z}^p\, dx+\int_{B_{R_1}}\overline{z}^p\, dx\ge\int_{\Omega_0\setminus\overline{\Theta}} w_{\Omega_0}^p\,dx+\int_{\Theta} w_{\Omega_0}^p\,dx+ z_m^p(\abs{B_{R_2}}-\abs{\Omega_{0}}).
$$
We recall that $\overline{z}=z$ on $A_{R_1,R_2}$, $\overline{z}=z_m$ in $B_{R_1}$ and $w_{\Omega_0}\ge z_m$ in $\Omega_0$ (so, in particular, on $\Theta$). We thus conclude
\begin{align*}
\int_{A_{R_1,R_2}}z^p\, dx&\ge\int_{\Omega_0\setminus\overline{\Theta}} w_{\Omega_0}^p\,dx+\int_{\Theta} w_{\Omega_0}^p\,dx-z_m^p\abs{\Omega_{0}}+z_m^p\abs{B_{R_2}}-z_m^p\abs{B_{R_1}}\\
&\ge \int_{\Omega_0\setminus\overline{\Theta}} w_{\Omega_0}^p\,dx+z_m^p\left(|B_{R_2}|-|B_{R_1}|-|\Omega_0|+|\Theta|\right)\\
&=\int_{\Omega_0\setminus\overline{\Theta}} w_{\Omega_0}^p\,dx+z_m^p\left(|A_{R_1,R_2}|-|\Omega_0\setminus\overline{\Theta}|\right)=\int_{\Omega_0\setminus\overline{\Theta}} w_{\Omega_0}^p\,dx,
\end{align*}
i.e. we have \eqref{zpbuco}. Moreover, \eqref{dzp} easily implies \eqref{dzpbuco}:
    \begin{align*}
  \int_{\Omega_0\setminus\overline{\Theta}} \abs{\nabla w_{\Omega_0} }^p \, dx  &\le\int_{\Omega_0} \abs{\nabla w_{\Omega_0} }^p \, dx\le\int_{B_{R_2}} \abs{\nabla \overline{z}}^p\,dx=\int_{A_{R_1,R_2}} \abs{\nabla z}^p\,dx.
    \end{align*}
We achieve the thesis plugging \eqref{zq}, \eqref{zpbuco} and \eqref{dzpbuco} in \eqref{pq_min}:
\begin{align*}
\tilde{\sigma}_{p,q}(A_{R_1,R_2})&=\dfrac{\left(\ds\int _{A_{R_1,R_2}}|\nabla z|^p\;dx+\ds\int_{A_{R_1,R_2}}z^p\;dx\right)^{1/p}}{\left(\ds\int_{\partial B_{R_2}}z^q \;d\mathcal{H}^{d-1}\right)^{1/q}}\ge\dfrac{\left(\ds\int _{\Omega_0\setminus\overline{\Theta}}|\nabla w|^p\;dx+\ds\int_{\Omega_0\setminus\overline{\Theta}}w^p\;dx\right)^{1/p}}{\left(\ds\int_{\partial \Omega_0}w^q \;d\mathcal{H}^{d-1}\right)^{1/q}}\\
&\ge\tilde{\sigma}_{p,q}(\Omega_0\setminus\overline{\Theta})
\end{align*}
We conclude observing that all the previous inequalities become equalities if and only if $\Omega_0\setminus\overline{\Theta}=A_{R_1,R_2}$.
\end{proof}

A crucial difference with the convex case in Proposition \ref{pro:gferone} is the following. Proposition \ref{pro:gferone} actually gives a quantitative result in view of the nonnegative term $|\Omega^\star|-|\Omega|$ that vanishes if and only if $\Omega$ is a ball itself. If $\Omega$ is nearly spherical, $|\Omega^\star|-|\Omega|$ can be estimated from below with the Hausdorff asymmetry in view of Lemma \ref{lemma_fugl_mod}. Instead, in the proof of Main Theorem 3, the volume constraint $|\Omega_0\setminus\overline{\Theta}|=|A_{R_1,R_2}|$ entails a cancellation of the additional volume term passing from \eqref{zp} to \eqref{zpbuco}. 

As highlighted in Section \ref{sec:hybrid}, in the linear case $p=q=2$, we overcome the problem of estimating the asymmetry of $\Omega_0$ using a Fuglede approach. To do this, in \cite{cpp}, authors adapted to the framework of domains with holes the argument of \cite{cito2021quantitative}, originally written in the framework of convex sets. A quantitative estimate of $\tilde{\sigma}_2^2(A_{R_1,R_2})-\tilde{\sigma}_2^2(\Omega_0\setminus\overline{\Theta})$ is then obtained in terms of the Hausdorff asymmetry of $\Omega_0$, provided that it is nearly spherical. In other words, if $\Omega_0$ is nearly spherical, for the linear eigenvalue $\tilde{\sigma}_2^2$ one has 
$$\tilde{\sigma}_2^2(A_{R_1,R_2})-\tilde{\sigma}_2^2(\Omega_0\setminus\overline{\Theta})\ge Cg(\mathcal{A}^\star_\mathcal{H}(\Omega_0))$$
for some constant depending only on $d,R_1,R_2$. 

In the nonlinear case, the idea is to use a propagation of the asymmetry argument; anyway, there are some relevant differences with respect to the classical cases. Firstly, here we deal with a maximization problem, not a minimization as in the Dirichlet case studied, for instance, in the classical references \cite{bdph,hansen1994isoperimetric}. In addition, we need to compare the perimeters of two sublevel sets  which do not match either a perimeter or a volume constraint; they are only linked by the "dearrangement" procedure and by the corresponding level $t$. In this sense, the result of Lemma \ref{pro:pro} will be very useful.

\begin{proposition}[outer asymmetry]\label{pro:buco2}
Let $0<R_1<R_2<+\infty$. Then, for every admissible set $\Omega_0\setminus\overline{\Theta}\subset\R^d$ such that $P(\Omega_0)=P(B_{R_2})$ and $|\Omega_0\setminus\overline{\Theta}|=|A_{R_1,R_2}|$,
it holds
\begin{equation*}
\tilde{\sigma}_{p,q}(A_{R_1,R_2})-\tilde{\sigma}_{p,q}(\Omega_0\setminus\overline{\Theta})\ge C(d,p,q,R_1,R_2)\tilde{\alpha}_{out}(\Omega_0).
\end{equation*} 
\end{proposition}
\begin{proof}
The proof is based on a refinement of the comparison argument of Main Theorem 3, plugging \eqref{eq:quantinner} instead of the simple inequality between perimeters. We set
$$\mu(t):=|\Omega_{0,t}|,\ \nu(t):=|A_{0,t}|.$$
It holds
 \begin{align*}
       -\nu'(t)=\frac{P(A_{0,t})}{\ell(t)}\ge \frac{P(\Omega_{0,t})}{\ell(t)}=-\mu'(t)
    \end{align*}
a.e. in $[z_m, z_M]$ and, in particular, when \eqref{eq:quantinner} holds
\begin{align*}
-\nu'(t)&=\frac{P(A_{0,t})}{\ell(t)}\ge \frac{P(\Omega_{0,t})}{\ell(t)}+C(d) d_\mathcal{H}^{\frac{d+3}{2(d-1)}}(\Omega_0;B_{\Omega_0})\frac{\tilde{G}^{-1}(t)}{\ell(t)}\\
&=-\mu'(t)+C(d) d_\mathcal{H}^{\frac{d+3}{2(d-1)}}(\Omega_0;B_{\Omega_0})\frac{\tilde{G}^{-1}(t)}{\ell(t)}
\end{align*}
a.e. in $[T(\Omega_0), z_M]$. Notice that, in view of its definition, it holds
$$\frac{\tilde{G}^{-1}(t)}{\ell(t)}=-\frac{d}{dt}\left[\frac{(\tilde{G}^{-1}(t))^2}{2}\right].$$
Integrating, we get $\nu(t)\ge\mu(t)+|\Omega_0|-|B_{R_2}|$ for any $z_m \leq t < z_M$. In particular, for any $t\in[T(\Omega_0),z_M]$:
\begin{equation*}
\begin{split}
\nu(t)-|B_{R_2}|&=-\int_t^{z_M}\nu'(s)\:ds\\
&=-\int_t^{z_M}\mu'(s)\:ds-C(d) d_\mathcal{H}^{\frac{d+3}{2(d-1)}}(\Omega_0;B_{\Omega_0})\int_t^{z_M}\frac{d}{ds}\left[\frac{(\tilde{G}^{-1}(s))^2}{2}\right]\:ds\\
&=\mu(t)-|\Omega_0|+C(d) d_\mathcal{H}^{\frac{d+3}{2(d-1)}}(\Omega_0;B_{\Omega_0})\left[\frac{(\tilde{G}^{-1}(t))^2}{2}\right].
\end{split}
\end{equation*}
In other words, for any $t\in[z_m,z_M]$:
\begin{equation}\label{eq:bucout1}
 \nu(t)\ge\mu(t)+|B_{R_2}|-|\Omega_0|+C(d) d_\mathcal{H}^{\frac{d+3}{2(d-1)}}(\Omega_0;B_{\Omega_0})\left[\frac{(\tilde{G}^{-1}(t))^2}{2}\right]\chi_{[T(\Omega_0),z_M]}(t).
\end{equation}
Integrating \eqref{eq:bucout1} one has
\begin{equation}\label{eq:bucout2}
\int_{B_{R_2}}\overline{z}^p\:dx\ge\int_{\Omega_0}w^p_{\Omega_0}\:dx+z_m^p(|B_{R_2}|-|\Omega_0|)+C(d) d_\mathcal{H}^{\frac{d+3}{2(d-1)}}(\Omega_0;B_{\Omega_0})\int_{T(\Omega_0)}^{z_M}{(\tilde{G}^{-1}(t))^2}\:dt.
\end{equation}
We conclude estimating last integral. By definition of $T(\Omega_0)$ it holds
\begin{align*}
\int_{T(\Omega_0)}^{z_M}{(\tilde{G}^{-1}(t))^2}\:dt&=\int_{T(\Omega_0)}^{z_M}\left(\int_t^{z_M}\frac{1}{\ell(s)}\:ds\right)^2\:dt\ge\frac{1}{4}\int_{T(\Omega_0)}^{z_M}\left(\int_t^{z_M}\frac{1}{\ell(z_M)}\:ds\right)^2\:dt\\
&=\frac{1}{4\ell(z_M)}\int_{T(\Omega_0)}^{z_M}\left(z_M-t\right)^2\:dt=\frac{(z_M-T(\Omega_0))^3}{12\ell(z_M)}.
\end{align*}
We thus use the latter estimate in \eqref{eq:bucout2} and split the plain part and the hole, obtaining
\begin{equation}\label{eq:bucout3}
\int_{A_{R_1,R_2}}z^p\:dx\ge\int_{\Omega}w_{\Omega_0}^p\:dx+\frac{1}{12\ell(z_M)}\tilde{\alpha}_{out}(\Omega_0),
\end{equation}
i.e. an enhanced version of \eqref{zpbuco}. Plugging \eqref{eq:bucout3} in the variational characterization of $\tilde{\sigma}_{p,q}$ and applying the Bernoulli's inequality we get:
\begin{align*}
\tilde{\sigma}_{p,q}(A_{R_1,R_2}) &=\left(\int_{A_{R_1,R_2}}|\nabla z|^p\:dx+\int_{A_{R_1,R_2}}z^p\:dx\right)^{1/p}\\
&\ge\left(\int_{\Omega_0\setminus \overline{\Theta}}|\nabla w_{\Omega_0}|^p\:dx+\int_{\Omega_0\setminus \overline{\Theta}}w^p_{\Omega_0}\:dx+\frac{1}{12\ell(z_M)}\tilde{\alpha}_{out}(\Omega_0)\right)^{1/p}\\
&\ge\left(\tilde{\sigma}^p_{p,q}(\Omega_0\setminus\overline{\Theta})+\frac{1}{12\ell(z_M)}\tilde{\alpha}_{out}(\Omega_0)\right)^{1/p}=\tilde{\sigma}_{p,q}(\Omega_0\setminus\overline{\Theta})\left(1+\frac{\tilde{\alpha}_{out}(\Omega_0)}{12\ell(z_M)\tilde{\sigma}^p_{p,q}(\Omega_0\setminus\overline{\Theta})}\right)^{1/p}\\
&\ge\tilde{\sigma}_{p,q}(\Omega_0\setminus\overline{\Theta})\left(1+\frac{\tilde{\alpha}_{out}(\Omega_0)}{12\ell(z_M)p\cdot\tilde{\sigma}^p_{p,q}(\Omega_0\setminus\overline{\Theta})}\right)\\
&=\tilde{\sigma}_{p,q}(\Omega_0\setminus\overline{\Theta})+\frac{\tilde{\alpha}_{out}(\Omega_0)}{12\ell(z_M)p\cdot\tilde{\sigma}^{p-1}_{p,q}(\Omega_0\setminus\overline{\Theta})}\\
&\ge\tilde{\sigma}_{p,q}(\Omega_0\setminus\overline{\Theta})+\frac{P(\Omega_0)^{\frac{p-1}{q}}}{12\ell(z_M)p|\Omega_0\setminus\overline{\Theta}|^{\frac{p-1}{p}}}\tilde{\alpha}_{out}(\Omega_0)\\
&=\tilde{\sigma}_{p,q}(\Omega_0\setminus\overline{\Theta})+\frac{ P(B_{R_2})^{\frac{p-1}{q}}}{12\ell(z_M)p|A_{R_1,R_2}|^{\frac{p-1}{p}}}\tilde{\alpha}_{out}(\Omega_0).
\end{align*}
Since, by \eqref{pqbuco_prob}, $\ell(z_M)^{p-1}=\tilde{\sigma}_{p,q}(A_{R_1,R_2})z_M^{q-1}$, the thesis is achieved.
\end{proof}

As highlighted in the introduction, the hole's deviation is estimated in terms of the deviation of $\Theta$ from the configuration where $\Omega_0\setminus\overline{\Theta}$ has the structure of a shell.

\begin{proposition}[inner asymmetry]
\label{teo:buco3}
Let $0<R_1<R_2<+\infty$. Then, for every $\Omega_0\setminus\overline{\Theta}\subset\R^d$ with $\Omega_0$ convex and $\Theta$ an admissible hole such that $\overline{\Theta}\subset\Omega_0$, $P(\Omega_0)=P(B_{R_2})$ and $|\Omega|=|A_{R_1,R_2}|$,
it holds
\begin{equation*}
\tilde{\sigma}_{p,q}(A_{R_1,R_2})-\tilde{\sigma}_{p,q}(\Omega_0\setminus\overline{\Theta})\ge C(d,p,q,R_1,R_2)\tilde{\mathcal{A}}(\Theta;\Omega_0),
\end{equation*}
where $\tilde{\mathcal{A}}(\Theta;\Omega_0)$ is defined in \eqref{eq:atilde}.
\end{proposition}
\begin{proof}
Let $z\in W^{1,p}(A_{R_1,R_2})$ be the minimizer for $\sigma_{p,q}(A_{R_1,R_2})$ with $\|z\|_{L^q(\partial B_{R_2})}=1$. Let us define the test function $w_{\Omega_0}\in W^{1,p}(\Omega_0)$ as in \eqref{GFerone} and denote $t_{\Omega_0},K_{\Omega_0}$ as in Definition \ref{inner_asymmetry_def}. We recall that $t_{\Omega_0}\ge R_2-R_1$ 
and so $w_{\Omega_0}=z_m$ in $K_{\Omega_0}$.

Notice that we can use $\Omega_0$ and as a test function for both $\Omega_0\setminus\overline{\Theta}$ and $\Omega_0\setminus\overline{K_{\Omega_0}}$, 
since $w_{\Omega_0}$ depends only on $\Omega_0$. We then get estimates \eqref{zp}, \eqref{dzp} and \eqref{zq} for the domain $\Omega_0\setminus\overline{K_{\Omega_0}}$:
\begin{equation}
    \label{eq:compnorm}
\begin{split}
\int_{\Omega_0\setminus\overline{K_{\Omega_0}}}|\nabla w_{\Omega_0}|^p\:dx\le\int_{A_{R_1,R_2}}|\nabla z|^p\:dx,\\
\int_{\Omega_0\setminus \overline{K_{\Omega_0}}}w^p_{\Omega_0}\:dx\le\int_{A_{R_1,R_2}}z^p\:dx,\\
\int_{\partial\Omega_0}w^q_{\Omega_0}\:d\mathcal{H}^{n-1}=\int_{\partial B_{R_2}}z^q\:d\mathcal{H}^{d-1}=1.
\end{split}
\end{equation}
The idea is to split smartly the volume integrals of over $\Omega_0$, choosing a different hole each time. We start with $\int_{\Omega_0}w_{\Omega_0}^p\:dx$. If we focus on the hole $\Theta$ we get
\begin{equation}
\label{eq:theta-kappa}
    \begin{split}     
\int_{\Omega_0}w_{\Omega_0}^p\:dx&=\int_{\Omega_0\setminus \overline{\Theta}}w_{\Omega_0}^p\:dx+\int_{\Theta\setminus K_{\Omega_0}}w_{\Omega_0}^p\:dx+\int_{K_{\Omega_0}\cap\Theta}w_{\Omega_0}^p\:dx.
    \end{split}
\end{equation}
On the other hand, using $K_{\Omega_0}$ as a hole,  \eqref{eq:compnorm} entails
\begin{equation}\label{eq:full}
\begin{split}
\int_{\Omega_0}w_{\Omega_0}^p\:dx&=\int_{\Omega_0\setminus \overline{K_{\Omega_0}}}w_{\Omega_0}^p\:dx+\int_{K_{\Omega_0}\setminus\Theta}w_{\Omega_0}^p\:dx+\int_{K_{\Omega_0}\cap\Theta}w_{\Omega_0}^p\:dx\\
&\le\int_{A_{R_1,R_2}}z^p\:dx+z_m^p|\Theta\setminus K_{\Omega_0}|+\int_{K_{\Omega_0}\cap\Theta}w_{\Omega_0}^p\:dx,
\end{split}
\end{equation}
where we have used that $w_{\Omega_0}=z_m$ in $K_{\Omega_0}$ and $|K_{\Omega_0}\setminus\Theta|=|\Theta\setminus K_{\Omega_0}|$. 
Putting together \eqref{eq:theta-kappa} and \eqref{eq:full}, we get 
\begin{equation}  \label{eq:quantinorm}
\int_{\Omega_0\setminus \overline{\Theta}}w^p_{\Omega_0}\:dx\le\int_{A_{R_1,R_2}}z^p\:dx-\int_{\Theta\setminus K_{\Omega_0}}(w_{\Omega_0}^p-z_m^p)\:dx.
\end{equation}
Arguing analogously for $\int_{\Omega_0\setminus \overline{\Theta}}|\nabla w_{\Omega_0}|^p\:dx$ we get
\begin{equation}  \label{eq:quantigrad}
\int_{\Omega_0\setminus \Theta}|\nabla w_{\Omega_0}|^p\:dx\le\int_{A_{R_1,R_2}}|\nabla z|^p\:dx-\int_{\Theta\setminus K_{\Omega_0}}|\nabla w_{\Omega_0}|^p\:dx.
\end{equation}
The conclusion follows plugging into \eqref{pqbuco_min} estimates \eqref{eq:compnorm}, \eqref{eq:quantinorm} and the equality in \eqref{eq:quantigrad}:
\begin{align*}
\tilde{\sigma}_{p,q}(A_{R_1,R_2}) &=\left(\int_{A_{R_1,R_2}}|\nabla z|^p\:dx+\int_{A_{R_1,R_2}}z^p\:dx\right)^{1/p}\\
&\ge\left(\int_{\Omega_0\setminus \overline{\Theta}}|\nabla w_{\Omega_0}|^p\:dx+\int_{\Theta\setminus K_{\Omega_0}}|\nabla w_{\Omega_0}|^p\:dx+\int_{\Omega_0\setminus \overline{\Theta}}w^p_{\Omega_0}\:dx+\int_{\Theta\setminus K_{\Omega_0}}(w_{\Omega_0}^p-z_m^p)\:dx\right)^{1/p}\\
&\ge\left(\tilde{\sigma}^p_{p,q}(\Omega_0\setminus\overline{\Theta})+\tilde{\mathcal{A}}(\Theta;\Omega_0)\right)^{1/p}=\tilde{\sigma}_{p,q}(\Omega_0\setminus\overline{\Theta})\left(1+\frac{\tilde{\mathcal{A}}(\Theta;\Omega_0)}{\tilde{\sigma}^p_{p,q}(\Omega_0\setminus\overline{\Theta})}\right)^{1/p}\\
&\ge\tilde{\sigma}_{p,q}(\Omega_0\setminus\overline{\Theta})\left(1+\frac{\tilde{\mathcal{A}}(\Theta;\Omega_0)}{p\cdot\tilde{\sigma}^p_{p,q}(\Omega_0\setminus\overline{\Theta})}\right)=\tilde{\sigma}_{p,q}(\Omega_0\setminus\overline{\Theta})+\frac{\tilde{\mathcal{A}}(\Theta;\Omega_0)}{p\cdot\tilde{\sigma}^{p-1}_{p,q}(\Omega_0\setminus\overline{\Theta})}\\
&\ge\tilde{\sigma}_{p,q}(\Omega_0\setminus\overline{\Theta})+\frac{P(\Omega_0)^{\frac{p-1}{q}}}{p|\Omega_0\setminus\overline{\Theta}|^{\frac{p-1}{p}}}\tilde{\mathcal{A}}(\Theta;\Omega_0)=\tilde{\sigma}_{p,q}(\Omega_0\setminus\overline{\Theta})+\frac{ P(B_{R_2})^{\frac{p-1}{q}}}{p|A_{R_1,R_2}|^{\frac{p-1}{p}}}\tilde{\mathcal{A}}(\Theta;\Omega_0).
\end{align*}
\end{proof}
Propositions \ref{pro:buco2} and \ref{teo:buco3} immediately entail the following
\begin{corollary}\label{cor:stabuco}
For any admissible set $\Omega_0\setminus\overline{\Theta}$, there exists a positive constant $C(d,p,q,R_1,R_2)$ such that
$$
\tilde{\sigma}_{p,q}(A_{R_1,R_2})-\tilde{\sigma}_{p,q}(\Omega_0\setminus\overline{\Theta})\ge C(d,p,q,R_1,R_2)\alpha_{hyb}(\Omega_0\setminus\overline{\Theta}).
$$
\end{corollary}

\subsection{Reduction to nearly annular sets}
Our aim is now to prove that the previous quantitative improvement can be stated in terms of stability of $\tilde{\sigma}_{p,q}$ in the class $\mathcal{T}_{R_1,R_2}$, defined in \eqref{admissiblesets}. The crucial point is to show that $\tilde{\sigma}_{p,q}$ takes values close to the optimum if and only if $\Omega_0\setminus\overline{\Theta}\in\mathcal{T}_{R_1,R_2}$ is nearly annular. In order to obtain it, we proceed as in \cite{cpp}, Section 4.3, showing to the readers, for their convenience, some of the proofs we deem more relevant.

We start by giving the so-called "isodiametric control" of $\tilde{\sigma}_{p,q}$. In other words, we show that if $\tilde{\sigma}_{p,q}(\Omega_0\setminus\overline{\Theta})$ is not much smaller than $\tilde{\sigma}_{p,q}(A_{R_1,R_2})$, then $diam(\Omega_0)$ is controlled by a uniform constant.


Isodiametric controls of eigenvalues, like the type discussed here, are a standard tool in shape maximization problems. This control usually contributes to compactness when tackling global existence problems (consider, for instance, the isodiametric control of the Robin spectrum detailed in \cite{bucurcito} and of the Steklov spectrum presented in \cite{bbg}, both of which are valid even for higher eigenvalues across a broader range of sets).

It is worth noting that the subsequent result extends to holed domains beyond the class $\mathcal{T}_{R_1,R_2}$. Consequently, we provide the proof in a more general context, as it is trivial for domains belonging to $\mathcal{T}_{R_1,R_2}$.

\begin{lemma}\label{isod}
Let $0<R_1<R_2<+\infty$. There exists a positive constant $C(d,p,q,R_1,R_2)$ such that, for every $\Omega_0\setminus\overline{\Theta}\subset\R^d$ with $\Omega_0$ and $\Theta$ open convex sets such that $\overline{\Theta}\subset\subset\Omega_0$, $P(\Omega_0)=P(B_{R_2})$ and $|\Omega_0\setminus\overline{\Theta}|=|A_{R_1,R_2}|$, if 
$$
\tilde{\sigma}_{p,q}(\Omega_0\setminus\overline{\Theta})\ge\frac{ \tilde{\sigma}_{p,q}(A_{R_1,R_2})}{2},
$$
then
$$\text{diam}(\Omega_0)\le C(d,p,q,R_1,R_2).$$
\begin{proof}
The proof is a straightforward adaptation to our context of the analogous result for the Robin-Neumann eigenvalues, see \cite[Lemma 4.4]{cpp} (see \cite[Lemma 3.5]{cito2021quantitative} for the proof for the Robin eigenvalues in the convex case).

It is based on the following contradiction argument. If there exists a sequence ${\Omega_0}_j\setminus\overline{\Theta_j}$ such that 
$$\tilde{\sigma}_{p,q}({\Omega_0}_j\setminus\overline{\Theta_j})\ge \frac{\tilde{\sigma}_{p,q}(A_{R_1,R_2})}{2}\quad \text{and}\quad\text{diam}({\Omega_0}_j)\to+\infty,$$
the convexity of ${\Omega_0}_j$ and the constraint $P({\Omega_0}_j)=P(B_{R_2})$, together with inequality 
$$|A|\le\rho P(A),$$
where $\rho$ is the inradius of the convex set $A$, (see, for instance, \cite[Prop. 2.4.3]{bucur2004variational}), would imply
that $|{\Omega_0}_j|$ vanishes as $j$ goes to $+\infty$. Now, using the characteristic function of ${\Omega_0}_j\setminus\overline{\Theta_j}$ as a test, we obtain (see Remark \ref{trivial})
$$\frac{\tilde{\sigma}_{p,q}(A_{R_1,R_2})}{2}\le\tilde{\sigma}_{p,q}({\Omega_0}_j\setminus\overline{\Theta_j})\le\frac{|{\Omega_0}_j\setminus\overline{\Theta_j}|^{1/p}}{P({\Omega_0}_j)^{1/q}} \to 0,$$
a contradiction.
\end{proof}
\end{lemma}


Some form of continuity or semicontinuity is required to simplify the problem to nearly annular sets. To achieve this, we prove the upper semicontinuity of $\tilde{\sigma}_{p,q}$ with respect to Hausdorff convergence. We start by establishing the lower semicontinuity of the boundary integral. For the special case $p=q=2$, see \cite[Lemma 4.5]{cpp}; the general case can be proved, for example, adapting \cite[Lemma 4.2]{bugitr16} to our case.

\begin{lemma}\label{pro_lscquadro}
Let $p>1$, $1<q<p^*$ and $E_j,E\subset\R^d$ be bounded convex domains such that $E_j\to E$ in the sense of Hausdorff. Let $w_j,w\in W^{1,p}(\R^d)$. If $w_j\rightharpoonup w$ in $W^{1,p}(\R^d)$, then it holds
$$\int_{\partial E} w^q\:d\mathcal{H}^{d-1}\le \liminf_{j}\int_{\partial E_j} w_j^q\:d\mathcal{H}^{d-1}$$
\end{lemma}
We are now in a position to prove the upper semicontinuity for $\tilde{\sigma}_{p,q}$.
\begin{lemma}\label{usc_lambda}
Let ${\Omega_0}_j\setminus\overline{\Theta_j},{\Omega_0}\setminus\overline{\Theta}\in\mathcal{T}_{R_1,R_2}$,  with ${\Omega_0}_j\setminus\overline{\Theta_j}\to{\Omega_0}\setminus\overline{\Theta}$ in the sense of Hausdorff (and in measure). Then 
$$
\limsup_{j\to+\infty} \tilde{\sigma}_{p,q}({\Omega_0}_j\setminus\overline{\Theta_j})  \leq\tilde{\sigma}_{p,q}({\Omega_0}\setminus\overline{\Theta}).
$$
\begin{proof}
Let $u\in W^{1,p}({\Omega_0}\setminus\overline{\Theta})$ a minimizer for $\tilde{\sigma}_{p,q}({\Omega_0}\setminus\overline{\Theta})$, and let $\overline{u}\in W^{1,p}(\R^d)$ an extension of $u$ to the whole of $\R^d$. Note that $\overline{u}$ is a test function for ${\Omega_0}_j\setminus\overline{\Theta_j}$ for every $j\in\N$. Moreover, under the given hypotheses, the convex sets ${\Omega_0}_j$ converge to $\Omega_0$
in the Hausdorff and measure senses. Hence, we have
$$
\int_{\partial \Omega_0}u^q d\H^{d-1}=\int_{\partial \Omega_0}\overline{u}^q d\H^{d-1}\le\liminf_{j\to+\infty}\int_{\partial {\Omega_0}_j}\overline{u}^q d\H^{d-1}
$$
as a consequence of Lemma \ref{pro_lscquadro}. Moreover, both volume integrals are continuous in view of the convergence in measure $\Omega_j\to\Omega$ and this entails the upper semicontinuity of the quotients in \eqref{pqbuco_min}. Therefore, we have
\begin{align*}
\tilde{\sigma}_{p,q}({\Omega_0}\setminus\overline{\Theta})&=\dfrac{\ds\left(\int _{\Omega_0\setminus\overline{\Theta}}|\nabla u|^p\;dx+\ds\int_{\Omega_0\setminus\overline{\Theta}}|u|^p\;dx\right)^{1/p}}{\left(\ds\int_{\partial \Omega_0}|u|^q \;d\mathcal{H}^{d-1}\right)^{1/q}}\\
&\ge\limsup_{j\to+\infty}\dfrac{\ds\left(\int _{{\Omega_0}_j\setminus\overline{\Theta_j}}|\nabla \overline{u}|^p\;dx+\ds\int_{{\Omega_0}_j\setminus\overline{\Theta_j}}|\overline{u}|^p\;dx\right)^{1/p}}{\left(\ds\int_{\partial{\Omega_0}_j}|\overline{u}|^q \;d\mathcal{H}^{d-1}\right)^{1/q}}
\ge\limsup_{j\to+\infty} \tilde{\sigma}_{p,q}({\Omega_0}_j\setminus\overline{\Theta_j}).
\end{align*}
\end{proof}
\end{lemma}

The uniqueness of $A_{R_1,R_2}$ as the maximizer of $\tilde{\sigma}_{p,q}(\cdot)$, coupled with the upper semicontinuity of $\tilde{\sigma}_{p,q}(\cdot)$ and the isodiametric control in Lemma \ref{isod}, ensures the following convergence result. The proof is not shown here, it is a actually contained in \cite[Lemma 4.7]{cpp}.

\begin{lemma}\label{lem:maxi}
Let $\{{\Omega_0}_j\setminus\overline{\Theta_j}\}_{j\in\N}\subset\mathcal{T}_{R_1,R_2}$ be a maximizing sequence for $\tilde{\sigma}_{p,q}$ with ${\Omega_0}_j\setminus\overline{\Theta_j}        $ is barycentered at the origin. Then  $$d_\H({\Omega_0}_j\setminus\overline{\Theta_j},A_{R_1,R_2}) \to 0.$$
\end{lemma}
The above lemma allows us to consider only nearly annular sets with their barycenter at the origin for our main stability result. The proof, based on a straightforward contradiction argument, is omitted here as it directly replicates \cite[Lemma 4.8]{cpp}.
\begin{lemma} \label{red}
Let $0<R_1<R_2<+\infty$. There exists a positive constant $\delta_0=\delta_0(n,\beta,R_1,R_2)$ such that, if ${\Omega_0}\setminus\overline{\Theta}\in\mathcal{T}_{R_1,R_2}$ and  $$\tilde{\sigma}_{p,q}(A_{R_1,R_2})-\tilde{\sigma}_{p,q}({\Omega_0}\setminus\overline{\Theta}) \leq \delta_0$$
then, up to a translation, ${\Omega_0}\setminus\overline{\Theta}$ is a $(R_1,R_2)-$nearly annular set.
\end{lemma}
Now that we have finally reduced ourselves to nearly annular sets, we can conclude the proof of Main Theorem 4.

\begin{proof}[Proof of Main Theorem 4]
It is an immediate consequence of Corollary \ref{cor:stabuco} and Lemma \ref{red}.
\end{proof}

\section{Further remarks and open problems}
\label{rem_sec}

In this section we collect some consequences of the previous reults and state some open problems arising from our analysis.

\subsection{Stability of the spherical shell for the Robin-Neumann problems with negative boundary parameter}

The techniques presented throughout the paper can be used to get analogous stability results for the spherical shell respectively for the Robin-Neumann eigenvalues with negative boundary parameter. We briefly recall the functional involved. Let us fix $\beta<0$. Given $\Omega_0\setminus\overline{\Theta}$ with $\Omega_0$ Lipschitz and $\Theta\subset\subset\Omega_0$ sufficiently smooth, the first Robin-Neumann eigenvalue of the $p$-Laplace operator with boundary parameter $\beta$ is the least number such that
\begin{equation*}
\begin{cases}
	-\Delta_p w=\tilde{\lambda}_p
(\beta, \Omega_0\setminus\overline{\Theta}) 
	|w|^{p-2}w  & \mbox{in}\ \Omega_0\setminus\overline{\Theta}\vspace{0.2cm}\\
|\nabla w|^{p-2}\dfrac{\partial w}{\partial \nu}+\beta |w|^{p-2} w=0&\mbox{on}\ \partial \Omega_0\vspace{0.2cm}\\ 
\dfrac{\partial w}{\partial \nu}=0&\mbox{on}\ \partial\Theta;
\end{cases}
\end{equation*}
$\tilde{\lambda}_p(\beta, \Omega)<0$ and it can be characterized in a variational way as follows: 
\begin{equation*}
\tilde{\lambda}_p
(\beta, \Omega_0\setminus\overline{\Theta})=\min_{\substack{\psi\in W^{1,p}(\Omega_0\setminus\overline{\Theta})\\ \psi\not \equiv0}}\frac{\ds\int_{\Omega_0\setminus\overline{\Theta}}|\nabla \psi|^p\:dx+\beta\int_{\partial\Omega_0}|\psi|^p\:d\mathcal{H}^{d-1}}{\ds\int_{\Omega_0\setminus\overline{\Theta}}|\psi|^p\:dx}.
\end{equation*}

The following result can be obtained repeating verbatim the arguments in Section \ref{main_sec_buco}, replacing in the definition of $\alpha_{hyb}$ the $L^q(\partial B_{R_2})$-normalized positive minimizer $z$ for $\tilde{\sigma}_{p,q}(A_{R_1,R_2})$ with the first positive $L^p(A_{R_1,R_2})$-normalized eigenfunction of $A_{R_1,R_2}$. Moreover, Theorem \ref{teo:robinp} gives a first positive answer to \cite[Open Problem 5.1]{cpp}.

\begin{theorem}\label{teo:robinp}
Let $0<R_1<R_2<+\infty$ and $p\in]1,+\infty[$. There exist two positive constants $C(p,q,d,R_1,R_2)$ and $\delta_0(p,q,d,R_1,R_2)$ such that, for every open set $\Omega_0\setminus\overline{\Theta}\in\mathcal{T}_{R_1,R_2}$, if $\tilde{\lambda}_p(\beta,A_{R_1,R_2})-\tilde{\lambda}_p(\beta,\Omega_0\setminus\overline{\Theta})\leq \delta_0$, then $\Omega_0\setminus\overline{\Theta}$ is $(R_1,R_2)$-nearly annular and
\begin{equation*}
\tilde{\lambda}_p(\beta,A_{R_1,R_2})-\tilde{\lambda}_p(\beta,\Omega_0\setminus\overline{\Theta}) \geq C(d,p,\beta,R_1,R_2)\alpha_{hyb}(\Omega_0\setminus\overline{\Theta})
.
\end{equation*}
\end{theorem}

\subsection{Open problems}
We conclude showing some possible perspective of research.

\begin{open} The question about the optimal geometries for $\sigma_{p,q}$ and $\tilde{\sigma}_{p,q}$ remains open regarding the case where $p<q<p^*$ and the minimizers for the associated variational problems on balls and spherical shells are not radial. In that scenario, such functions are the product of a radial component and of a component depending only on the distance from the north pole (see Section 5 in \cite{DoTo}). It would be interesting to understand what the expected maximizing shape is, at least numerically. If the expected maximum is still the ball or the spherical shell, it would be interesting to understand how to set up a dearrangement technique without having the radiality of the minimizers for the associated variational problems.
\end{open}

\begin{open}
The choice of constraining the outer perimeter and the volume of the holed domains $\Omega_0\setminus\overline{\Theta}$ for maximizing $\tilde{\sigma}_{p,q}$ is rather common and dates back to the pioneeristic work \cite{payne1961some}. The choice of fixing the perimeter of the convex set $\Omega$ in the maximization of $\sigma_{p,q}$ is \emph{trace-focused} and is naturally suggested by several  spectral maximization problems in which a competition between volume and surface integrals appears. Among them it is worth to recall, for instance, \cite{amato2024estimates,bfnt,cito2021quantitative} for the first (linear or nonlinear) Robin eigenvalue with negative boundary parameter and \cite{gavitone2020quantitative} for the first nontrivial Steklov eigenvalue. In both the frameworks of convex sets and domains with holes, the web function method naturally accommodates the choice of admissible sets and the constraints. It would be interesting to understand what happens modifying the constraints (e.g., maximizing $\sigma_{p,q}$ among convex sets with fixed \emph{measure}) or enlarging the class of admissble sets (e.g., starshaped domains).
\end{open}

\begin{open}
The hybrid asymmetry $\alpha_{hyb}(\cdot)$ is actually a distance from the spherical shell and enjoys the remarkable property of recognizing the more symmetric between two holed sets that have the same containter $\Omega_0$ and different holes, e.g. if they have the same inner set $\Theta$ in different relative positions. Anyway, we are aware that it is not immediate to understand its meaning and that it is not very elegant or clear at a first sight. Even its sharpness seems far from being true. For this reason, an interesting challenge could be to find a way to improve the asymmetry functionals for holed domains introduced here and in \cite{cpp}.
\end{open}

\section*{Acknowledgment}
S. Cito was supported by "Elliptic and parabolic problems, heat kernel estimates and spectral theory" PRIN 2022 project n. 20223L2NWK, funded by the Italian Ministry of University and
Research.

The author was supported by Gruppo Nazionale per l’Analisi
Matematica, la Probabilità e le loro Applicazioni (GNAMPA) of Istituto Nazionale di Alta
Matematica (INdAM).

\bibliographystyle{abbrvurl}
\bibliography{bibliography}
\end{document}